\documentclass[11pt, letterpaper]{amsart}    	% use "amsart" instead of "article" for AMSLaTeX format
\usepackage[margin=1in]{geometry}
\usepackage{setspace}
\usepackage{graphicx}	
\usepackage{subcaption}
\usepackage{float}
\usepackage[toc]{appendix}
\usepackage{bm}
\usepackage[backref=page]{hyperref}
\usepackage{amsfonts}
\usepackage{dsfont}

\usepackage{enumitem}
\usepackage{lineno}

\usepackage{amssymb}
\usepackage{amsmath, amsthm, amssymb, mathrsfs}

\newtheorem{thm}{Theorem}[section]
\newtheorem{conj}[thm]{Conjecture}

\newtheorem{lemma}[thm]{Lemma}

\newtheorem{prob}[thm]{Problem}
\newtheorem{observation}[thm]{Observation}

\theoremstyle{definition}
\newtheorem{defi}[thm]{Definition}

\def\ex{\text{ex}}
\def\cover{{\bf COVER}}
\def\factor{{\bf FACTOR}}
\def\trans{{\bf TRANS}}
\usepackage{tikz}
\tikzstyle{aNode} = [circle, fill = black]
\tikzstyle{bNode} = [circle,draw = black, thick]

\newcommand{\ppoints}[1]{%
\begin{tikzpicture}[inner sep = 0.7pt, #1]%
\node (1) at (0,-2) [aNode]{};
\node (3) at (1.5,-2) [aNode]{};
\node (2) at (0.75,-1) [aNode]{};
\end{tikzpicture}%
}

\def\points{\ppoints{scale=0.11}}

\newcounter{casenum}

\newcommand*{\rom}[1]{\expandafter{\romannumeral #1\relax}}

\begin{document}
%\linenumbers
\title[$F$-factors in Quasi-random Hypergraphs]{$F$-factors in Quasi-random Hypergraphs}
\date{\today}
\author{Laihao Ding}
\author{Jie Han}
\author{Shumin Sun}
\author{Guanghui Wang}
\author{Wenling Zhou}

\address
{School of Mathematics and Statistics, and Hubei Key Laboratory of Mathematical Sciences, Central China Normal University,
Wuhan, China.}
\thanks{LHD is supported by the National Natural Science Foundation of China (11901226) and the China Postdoctoral Science Foundation (2019M652673).
JH is partially supported by Simons Collaboration Grant \#630884.
GHW is supported by the Natural Science Foundation of China (11871311) and  Young Taishan Scholars program of Shandong Province (201909001).}
\email[Laihao Ding]{dinglaihao@ccnu.edu.cn}
\address
{School of Mathematics and Statistics, Beijing Institute of Technology, Beijing, China.}
%\thanks{JH is partially supported by Simons Collaboration Grant \#630884.}
\email[Jie Han]{hanjie@bit.edu.cn}
\address
{School of Mathematics, Shandong University, Jinan, China.}
\email[Shumin Sun]{sunshumin987@163.com}
%\thanks{GHW is partially supported by the Natural Science Foundation of China  (11871311, 11631014) and Shandong University multidisciplinary research and innovation team of young scholars.}
\email[Guanghui Wang]{ghwang@sdu.edu.cn}
\address
{School of Mathematics, Shandong University, Jinan, China, and Laboratoire Interdisciplinaire des Sciences du Num\'{e}rique, Universit\'{e} Paris-Saclay, France.}
\email[Wenling Zhou]{gracezhou@mail.sdu.edu.cn}
\date{\today}
%\begin{titlepage}
%\title{$F$-factors in Linear Quasi-random Hypergraphs}
%\author{Laihao Ding}
%\thanks{ LHD is supported by the National Natural Science Foundation of China (11901226). Email: \texttt{dinglaihao@mail.ccnu.edu.cn}.}
%\author{Jie Han}
%\thanks{Department of Mathematics, University of Rhode Island, RI, USA. JH is partially supported by Simons Collaboration Grant \#630884. Email: \texttt{jie\_han@uri.edu}.}
%\author{Shumin Sun}
%\thanks{School of Mathematics, Shandong University,
%Jinan, China. Email: \texttt{sunshumin987@163.com}.}
%\author{Guanghui Wang}
%\thanks{School of Mathematics, Shandong University,
%Jinan, China. GHW is partially supported by the Natural Science Foundation of China  (11871311, 11631014) and
%Shandong University multidisciplinary research and innovation team of young scholars.
%Email: \texttt{ghwang@sdu.edu.cn}.}
%\author{Wenling Zhou}
%\thanks{School of Mathematics, Shandong University,
%Jinan, China. Email: \texttt{gracezhou@mail.sdu.edu.cn}.}
%
%\date{\today}

\begin{abstract}
Given $k\ge 2$ and two $k$-graphs ($k$-uniform hypergraphs) $F$ and $H$, an \emph{$F$-factor} in $H$ is a set of vertex-disjoint copies of $F$ that together cover the vertex set of $H$.
Lenz and Mubayi [\textit{J. Combin. Theory Ser.~B, 2016}] studied the $F$-factor problem in quasi-random $k$-graphs with minimum degree $\Omega(n^{k-1})$. They posed the problem of characterizing the $k$-graphs $F$ such that every sufficiently large quasi-random $k$-graph with constant edge density and minimum degree $\Omega(n^{k-1})$ contains an $F$-factor, and in particular, they showed that all linear $k$-graphs satisfy this property.

In this paper we prove a general theorem on $F$-factors which reduces the $F$-factor problem of Lenz and Mubayi to a natural sub-problem, that is, the $F$-cover problem.
%By using this result, we give a characterization theorem of such $F$ when $F$ is a $k$-partite $k$-graph, $k\ge 3$.
%Finally, we give a characterization result of such $F$ for $3$-graphs, namely, we solve the problem of Lenz and Mubayi for $k=3$.
%Finally, motivated by the recent work of Reiher, R\"odl and Schacht [\textit{J.~Lond.~Math.~Soc., 2018}] that classifies the 3-graphs with vanishing Tur\'an density in quasi-random $k$-graphs, we give a complete solution to the problem of Lenz and Mubayi for $k=3$.
%
By using this result, we answer the question of Lenz and Mubayi for those $F$ which are $k$-partite $k$-graphs, and for all 3-graphs $F$, separately.
Our characterization result on 3-graphs is motivated by the recent work of Reiher, R\"odl and Schacht [\textit{J.~Lond.~Math.~Soc., 2018}] that classifies the 3-graphs with vanishing Tur\'an density in quasi-random $k$-graphs.

\end{abstract}

\maketitle
%\medskip

%\noindent {\bf Keywords:} quasi-random hypergraph; $F$-factor; absorbing method; hypergraph regularity method
%\end{titlepage}

\section{Introduction}
Given $k\geq 2$, a \textit{$k$-uniform hypergraph} $H$ ($k$-graph for short) is a pair $H=(V(H),E(H))$ where $V(H)$ is a finite set of \textit{vertices} and $E(H)=\{e\subseteq V(H): |e|=k\}$ is a set of $k$-element subsets of $V(H)$, whose members are called the \textit{edges} of $H$. As usual $2$-graphs are simply called graphs.

Given a fixed graph $F$, a typical problem in extremal graph theory asks for the maximum number of edges that a large graph $G$ on $n$ vertices containing no copy of $F$ can have, denoted by $\ex(n, F)$.
The systematic study of such problems was initiated by Tur\'an~\cite{Turan}, who determined $\ex(n, K_k)$ for complete graphs $K_k$.
Targeting on the limit behavior, one may define the \emph{Tur\'an density}, $\pi(F):=\lim_{n\to\infty}\ex(n, F)/\binom n2$.
For arbitrary $F$, the Tur\'an density $\pi(F)$ was determined by Erd\H{o}s and Stone~\cite{P1963On}.
The Tur\'an problem for $k$-graphs is notoriously difficult -- despite much efforts and attempts so far, the Tur\'an density of tetrahedron, the smallest clique in 3-graphs, raised by Tur\'an in 1941, is still open~\cite{P1977On, Razborov2010}.
%In his original work~\cite{Turan}, Tur\'an asked for hypergraph extensions of these extremal problems.
%However, the Tur\'an problems for hypergraphs appear to be very difficult.
%In particular, the Tur\'an density of tetrahedron (conjectured to be $5/9$), the smallest clique in 3-graphs, is only known to be between $5/9$ and $0.5616$, where the lower bounds are given by what is believed to be optimal constructions due to
%Tur\'an~\cite{P1977On} and the stated upper bound is due to Razborov~\cite{Razborov2010}.

An important reason for the extreme difficulty in the hypergraph Tur\'an problems is the existence of certain \emph{quasi-random} construction obtained from random tournaments or random colorings of $E(K_n)$.
This behavior is quite different from the graph case as quasi-random graphs with positive density contain a correct number of copies of arbitrary graphs $F$ of fixed size, namely, the number of copies of $F$ is as expected in the random graph with the same density.
This suggests a systematic study of subgraph containment problems in quasi-random hypergraphs.
In fact, Erd\H{o}s and S\'os~\cite{Erd1982On} were the first to raise questions on the Tur\'an densities in ($p,\mu $)-dense 3-graphs (see Definition~\ref{def.}).
Their conjecture for the Tur\'an density of $K_4^{(3)-}$ (the unique 3-graph with 4 vertices and 3 edges) was solved only recently by Glebov, Kr\'al' and Volec~\cite{Glebov_2015}, and independently by Reiher, R\"odl and Schacht~\cite{Reiher_2018}.

\subsection{Quasi-random graphs and hypergraphs}
The study of quasi-random graphs was launched in late 1980s by Chung, Graham and Wilson~\cite{1988Quasi}.
These are constant density graphs which behave like the random graphs.
There is a list of properties that force a graph to be quasi-random and it was shown that these properties are all equivalent (up to a decay of the constants).
Among others, a notable application of quasi-random graphs is the upper bounds on the diagonal Ramsey numbers~\cite{ThomaAn,ConlonA,sah2020diagonal}.

%In the case of quasi-random graphs, the problem of subgraph containment is well understood~\cite{CGW,komlos1997blow} by the celebrated Blow-up lemma in graphs.
%For the sparse range, where the edge density depends on $n$, it turned out to be  notoriously difficult  already  for small graphs~$F$, and even more so for spanning subgraphs.

The investigation of quasi-random $k$-graphs was started by Chung and Graham~\cite{chung1990quasi} and is widely popular   \cite{quasihyper,Ch90,Ch91,Chung10,wquasi,G06,G07,KNRS,KRS02,LenzMubayi_eig,lenz2015poset,RSk04,Towsner}.
As mentioned above, the subgraph containment problem for quasi-random $k$-graphs, $k\geq3$, is quite different from the case $k=2$ and
has been an important topic over decades.
Indeed, for $k\geq 3$, R\"odl noted that by a construction of~\cite{ErdosHajnal}, quasi-random $k$-graphs may not contain a single copy of, say, a $(k+1)$-clique.

The study of Tur\'an-type problems for quasi-random $k$-graphs was recently popularized by Reiher, R\"odl and Schacht~\cite{Reiher_2017}.
Besides a solution to the aforementioned conjecture of Erd\H{o}s and S\'os, they also determined a large collection of Tur\'an densities, under a family of naturally defined quasi-randomness conditions, see the recent survey of Reiher~\cite{Reiher_2020}.
%However, in this paper we will restrict ourselves to the considerably weak form of the following quasi-randomness, and we recommend the readers a detailed discussion of both the quasi-randomness and the results to the recent survey of Reiher~\cite{Reiher_2020}.

%We begin with some notation. For a positive integer $\ell$, we denote by $[\ell]$ the set $\{1,\dots,\ell\}$.
%Given $k\geq 2$, for a set $V$, we use $[V]^k$ to denote the collection of all subsets of $V$ of size $k$. We may drop
%one pair of brackets and write $[\ell]^k$ instead of $[[\ell]]^k$.
%A \emph{$k$-graph} $H^{(k)}$ (also known as a \textit{$k$-uniform hypergraph}), consists of a set of \emph{vertices} $V(H^{(k)})$ and a set of \emph{edges} $E(H^{(k)})\subseteq [V(H)]^k$. When the context is clear, we may drop the superscript and write $H$ instead of $H^{(k)}$. For convenience, let $v(H)=|V(H)|$ and $e(H)=|E(H)|$.
%Given a $k$-graph $H$ and a subset $S\subseteq V(H)$ with $|S|=s$, the degree of $S$, denoted by $\deg_H(S)$ or $\deg(S)$, is the number of edges of $H$ containing $S$ as a subset. Let $N_H(S)$ or $N(S)$ denote the neighbor set of $S$, i.e., $N_H(S)=\{ S' \in [V(H)]^{k-s} \mid S'\cup S\in E(H)\}$.
%The minimum $s$-degree $\delta_s(H)$ is the minimum of $\deg(S)$ over all $s$-subsets $S$ of $V(H)$. In particular, we call the minimum $1$-degree and the minimum $(k-1)$-degree of $H$ as the\emph{ minimum vertex degree} and \emph{minimum codegree} of $H$ respectively. If $S=\{v\}$ is a singleton then we write that $\deg_H(v)$ and $N_H(v)$ simply.

Here we state a notion of quasi-randomness considered by Lenz and Mubayi, which will also be the main notion discussed in this paper.

\begin{defi}[($p,\mu $){-denseness}] \label{def.}
Given integers $n\geq k\geq2$, let $0<\mu,p <1$, and $H$ be a $k$-graph with $n$ vertices. We say that $H$ is ($p,\mu $)\emph{-dense} if for all not necessarily disjoint $X_1,\dots,X_k\subseteq V(H)$,
\begin{equation}\label{eq:count}
e_H(X_1,\dots,X_k)\geq p|X_1|\cdots|X_k|-\mu n^k,
\end{equation}
where $e_{H}(X_1,\dots,X_k)$ is the number of $(x_1,\dots,x_k)\in X_1\times \cdots \times X_k$ such that $\{x_1,\dots,x_k\}\in E(H)$.
\end{defi}

%Given $k\geq 2$, a \textit{$k$-uniform hypergraph} $H$ ($k$-graph for short) is a pair $H=(V(H),E(H))$ where $V(H)$ is a finite set of \textit{vertices} and $E(H)=\{e\subseteq V(H)\mid |e|=k\}$ is a set of $k$-element subsets of $V(H)$, whose members are called the \textit{edges} of $H$. As usual $2$-graphs are simply called graphs. For convenience, let $v(H)=|V(H)|$ and $e(H)=|E(H)|$.

This is considerably weaker than other known quasi-randomness that have been discussed in the literature.
In particular, we call a $k$-graph $H$ weakly $(p,\mu )$\emph{-quasi-random} if for all not necessarily disjoint $X_1,\dots,X_k\subseteq V(H)$, $|e_H(X_1,\dots,X_k)- p|X_1|\cdots|X_k| |\le \mu n^k$.
The works of Kohayakawa {et al.}~\cite{KNRS} and Conlon {et al.}~\cite{wquasi}  show that weakly quasi-randomness is strongly related, in fact \emph{equivalent},
to the counting property of \emph{linear} $k$-graphs,
those in which any two edges intersect in at most one vertex.
%More precisely, a sequence of  $k$-graphs~$(H_n)_{n\to\infty}$, $n=v(H_n)$ is \emph{quasi-random} if and only if $H_n$ contains a correct number of labelled copies of {any} linear $k$-graph $F$.
% satisfies the counting property
%for all  linear $k$-graphs $F$ of constant size and conversely, if the counting property  holds for all
%linear $k$-graphs $F$ then $H$ must be quasi-random.
Due to this connection weakly quasi-randomness is often referred to as \emph{linear quasi-randomness}, and clearly the notion of quasi-randomness in Definition~\ref{def.} is even slightly weaker.
We omit a throughout discussion on the strengths of different notions of quasi-randomness.

We also omit a discussion on the so-called \emph{pseudo-random} (hyper)graphs, which are random-like graphs with density $p=p(n)$ tends to 0 when $n$ goes to infinity, and refers to the excellent survey~\cite{KS06} of Krivelevich and Sudakov.

\subsection{{\emph F}-factors in quasi-random $k$-graphs}
In this paper we focus on a strengthening of the Tur\'an-type problems, namely, the $F$-factor problems.
Given two $k$-graphs $F$ and $H$, a \emph{perfect $F$-tiling $($or $F$-factor$)$} in $H$ is a set of vertex-disjoint copies of $F$ that together covers the vertex set of $H$.
The study of perfect tilings in graph theory has a long and profound history, ranging from classical results back to Corradi--Hajnal~\cite{Corr} and Hajnal--Szemer\'edi~\cite{zbMATH03344609} on $K_k$-factors to the celebrated result of Johansson--Kahn--Vu~\cite{Johansson2008Factors} on perfect tilings in random graphs.
We note that the $F$-factor problems for quasi-random graphs with constant density has been solved
implicitly by Koml\'{o}s--S\'{a}rk\"{o}zy--Szemer\'{e}di~\cite{Koml1997Blow} in the course of developing the famous Blow-up Lemma.

%Our starting point then is the following result of Lenz and Mubayi, which is the first result on $F$-factors in quasi-random $k$-graphs, for $k\ge 3$.
Lenz and Mubayi~\cite{Lenz2016Perfect} were the first to study the $F$-factor problems in quasi-random $k$-graphs, for $k\ge 3$.
In particular, they raised the following basic question.
Throughout the paper, we say that $H$ is an $(n,p,\mu)$ \emph{$k$-graph} if $H$ has $n$ vertices and is ($p,\mu $)-dense; moreover, given $\alpha \in (0,1)$, an ($n,p,\mu$) $k$-graph is an ($n,p,\mu,\alpha$) \emph{$k$-graph} if its each vertex is contained in at least $\alpha {n}^{k-1}$ edges. For convenience, for a $k$-graph $H$, let $v(H)=|V (H)|$.
%Throughout the paper, given $p,\mu,\alpha \in (0,1)$, we call an $n$-vertex $k$-graph $H$ an ($n,p,\mu,\alpha$) $k$-graph if $H$ is ($p,\mu$)-dense and every vertex of $H$ is at least $\alpha {n}^{k-1}$ edges.

%\begin{thm}\label{2-graph}\cite{Koml1997Blow,Lenz2016Perfect}
%Let $0<\alpha, p <1$ be fixed and let $F$ be any graph. There exists an $n_0\in \mathbb{N} $ and $\mu > 0$ such that if $H$ is any ($n,p,\mu,\alpha$) $2$-graph where $n \geq n_0$, $v(F)|n$ then $H$ has an $F$-factor.
%\end{thm}

%Lenz and Mubayi~\cite{Lenz2016Perfect} extended Theorem \ref{2-graph} to a variety of hypergraphs, and asked a basic problem in this area.

\begin{prob}\label{prob}\cite{Lenz2016Perfect}~\label{prob}
Which $k$-graphs $F$ have the following property that for all $0< p,\alpha<1$, there is some $n_0\in \mathbb{N}$ and $\mu > 0$ so that if $H$ is any $(n,p,\mu,\alpha)$ $k$-graph with $n \geq n_0$ and $v(F)\mid n$, then $H$ has an $F$-factor?
\end{prob}

That is, Problem~\ref{prob} asks for a characterization of the $k$-graphs $F$ which allow degenerate (or vanishing) choices of $p$ and $\alpha$ to guarantee an $F$-factor in ($n,p,\mu,\alpha$) $k$-graphs.
Let $\factor_k$ be the collection of $k$-graphs $F$ satisfying the property above.
Lenz and Mubayi~\cite{Lenz2016Perfect} proved that all linear $k$-graphs belong to $\factor_k$ and provided a family of $3$-graphs not in $\factor_3$.
In particular, they showed that the complete $3$-partite $3$-graph $K_{2,2,2}$ is not in $\factor_3$ by providing a  parity-based construction~\cite[Theorem 5]{Lenz2016Perfect}.
At last, we quote from Lenz and Mubayi~\cite{Lenz2016Perfect} that ``solving Problem 2 (\emph{Problem~\ref{prob}}) will be a difficult project''.

An extension of their results on $F$-factors for linear $F$ to \emph{sparse quasi-random $k$-graphs} was obtained recently by H\`an, Han and Morris~\cite{HHM}.
%They also showed that two simple $3$-graphs are factorable.
%A \textit{cherry} is the $3$-graph comprising two edges that share precisely two vertices this is the ``simplest'' non-linear hypergraph.

%\begin{thm}\cite{Lenz2016Perfect}
%Given integer $k \geq 2$, all linear $k$-graphs are in \factor$_k$.
%\end{thm}
%

%\begin{thm}\cite{Lenz2016Perfect}
%Let $0<\alpha, p <1$. There exists an $n_0\in \mathbb{N}$ and $\mu > 0$ such that if $H$ is an ($n,p,\mu,\alpha$) $3$-graph with $n \geq n_0$, then $H$ has a perfect cherry-tiling if $4|n$ and a perfect $C_4(2+1)$-tiling if $6|n$.
%\end{thm}

\subsection{Our results}

%Our results are two-fold.
In this paper we mainly study Problem~\ref{prob}.
We first discuss an interesting phenomenon that appears in $F$-factors in hypergraphs and random graphs and give a significant reduction to the $F$-factor problems (see Theorem~\ref{k-graph}).
%Secondly, we solve Problem~\ref{prob} for $k=3$.

\subsubsection{Cover barriers and a reduction of Problem~\ref{prob}.}
We first note that a natural barrier for perfect tilings:
%When Han--Zang--Zhao~\cite{Han2017Minimum} studied perfect tilings in $3$-graphs, they discovered a natural barrier for perfect tilings:
\begin{quote}
    \emph{a trivial necessary condition for a $k$-graph $H$ to contain an $F$-factor is that every vertex of $H$ is covered by at least one copy of $F$}.
\end{quote}
%\emph{a necessary condition for a $k$-graph $H$ to contain an $F$-factor is that every vertex of $H$ must be covered by at least one copy of $F$}.
This can be visualized as a set of (not necessarily vertex-disjoint or distinct) copies of $F$ whose union covers the vertex set of the host graph, so also called an \emph{$F$-cover}.
In particular, this necessary condition leads to a minimum degree threshold for $F$-factor\footnote{Given $F$ and an integer $n$ divisible by $|V(F)|$, one can define $m(F, n)$ as the smallest integer $m$ such that every $k$-graph $H$ on $n$ vertices where every vertex is in at least $m$ edges contains an $F$-factor. Then the minimum degree threshold for $F$-factor is defined as $\lim_{n\to\infty}{m(F, n)}/{\binom{n-1}{k-1}}$.} for certain 3-partite 3-graphs $F$ to be \emph{irrational} (see~\cite{Han2017Minimum}).
Moreover, it plays an important role in the existence of $F$-factors in random (hyper)graphs as in the celebrated result of Johansson--Kahn--Vu (which results a necessary log factor in their result), but was always innocently satisfied for the study of $F$-factors in dense graphs.
One should also notice the increasing ``difficulty'' in terms of the assumptions imposed of the Tur\'an problem for $F$, the $F$-cover problem and the $F$-factor problem -- the latter ones are presumably harder.

%\footnote{It was definitely noticed, e.g. in the study of $F$-factors in random graphs, and it is the reason for the existence of the additional log factor in the Johansson--Kahn--Vu's result for strictly balanced graphs.}.
%The reason behind is that the Tur\'an density of a graph $F$ is the same as the density imposed by a minimum degree condition that guarantee the existence of $F$, namely, known extremal examples are usually vertex-transitive and once above the
For our problem, the ``cover'' property plays an essential role.
To be more precise, we first note that being $(p,\mu)$-dense puts no control on individual vertices -- indeed, there might be a small linear-sized set of vertices which are completely out of control, e.g., they could be isolated vertices, and that is the reason why a minimum degree condition needs to be imposed.
However, a minimum degree of form $\Omega(n^{k-1})$ would only allow us to embed a limited structure in the \emph{link $(k-1)$-graph} of a given vertex, namely, to embed $F$, there must be a vertex $v\in V(F)$ such that $N_F(v)$ is a $(k-1)$-partite $(k-1)$-graph, where $N_F(v)$ is the neighbor set of $v$.

This motivates the following definition:
%Given integers $f\geq k \geq 2$, we say that a $k$-graph $F$ with $f$ vertices is {\color{red}\textit{coverable}} if $F$ has the following property: For all $0< p,\alpha <1$, there is some $n_0$ and $\mu, \varepsilon> 0$ so that if $H$ is an ($n,p,\mu,\alpha$) $k$-graph with $n \geq n_0$, then each vertex $v$ in $V(H)$ is contained in at least $\varepsilon n^{f-1}$ copies of $F$.
given integers $f\geq k \geq 2$, let $\cover_k$ be the collection of $k$-graphs $F$ satisfying the following property: for all $0< p,\alpha <1$, there is some $n_0$ and $\mu> 0$ so that if $H$ is an ($n,p,\mu,\alpha$) $k$-graph with $n \geq n_0$, then each vertex $v$ in $V(H)$ is contained in a copy of $F$, namely, $H$ has an $F$-cover.
%In this paper, we will prove that all coverable $k$-graphs are factorable.

Interestingly, in this paper we show that  in $(n,p,\mu,\alpha)$ $k$-graphs, essentially \emph{being able to cover every vertex by a copy of $F$} is enough to guarantee the existence of $F$-factors, bypassing the parity constructions given by Lenz and Mubayi.
On the other hand, we trivially have $\factor_k\subseteq \cover_k$.
Altogether we obtain the following.

\begin{thm}\label{k-graph}
For $k \geq 2$, $\factor_k= \cover_k$.
\end{thm}

Theorem~\ref{k-graph} says that the $F$-factor problem in $(p,\mu)$-dense $k$-graphs stands close to that problem in binomial random $k$-graphs, namely, the Johansson-Kahn-Vu Theorem, where for both problems, the $F$-cover property appear as the ``bottle-neck'' -- a naive way to understand this phenomenon is to say that as long as the (quasi-)randomness forces the appearance of an $F$-cover, one can sort out the copies of $F$ and indeed have an $F$-factor.

Theorem~\ref{k-graph} serves as a general tool for the $F$-factor problem in $(p,\mu)$-dense $k$-graphs.
%It reduces the problem of $F$-factors to the study of ``covering a vertex of the host $k$-graph by a copy of $F$'', which is a natural generalization of the Tur\'an-type problem concerning ``rooted-embedding'', namely, embed a $k$-graph with a prescribed root.
It reduces the $F$-factor problem to the $F$-cover problem, which is a natural strengthening of the Tur\'an-type problem concerning ``rooted-embedding'', namely, embed a $k$-graph with a prescribed root vertex.
Moreover, this natural reduction eliminates the difficulty arising from divisibility issues (known as \emph{divisibility barriers} in the Dirac-type setting) and sheds light on the study of $F$-factors in quasi-random $k$-graphs.
Our two other main results (Theorems~\ref{k-partite} and~\ref{equal}) of this paper are then good examples for applications of Theorem~\ref{k-graph}.
%In particular, since we only assume $(p,\mu)$-denseness and the minimum degree condition, it can also be applied to $k$-graphs with \emph{stronger} quasi-randomness condition and degree assumptions.

\subsection{A characterization of $\factor_k$ for $k$-partite $k$-graphs $F$}
Recall that Lenz and Mubayi showed that all linear $k$-graphs are in $\factor_k$.
Linear $k$-graphs were the first class of $k$-graphs that was considered because $(p,\mu)$-denseness implies the correct counts of them.
To attack Problem~\ref{prob}, another natural class of $k$-graphs to consider is the \emph{$k$-partite $k$-graphs}.
A $k$-graph $H$ is called $k$-partite if its vertex set can be partitioned into $k$ parts, namely, $V(H)=V_1\cup V_2\cup \cdots \cup V_k$, such that $E(H)\subseteq V_1\times V_2\times \cdots \times V_k$.
One reason to consider $k$-partite $k$-graphs is that their Tur\'an densities are all zero by an old result of Erd\H{o}s~\cite{Onextremal}.

As an application of Theorem~\ref{k-graph}, we obtain the following characterization theorem for $F\in \factor_k$ when $F$ is $k$-partite. Given $k\geq 2$, for a set $V$, we use $[V]^k$ to denote the collection of all subsets of $V$ of size $k$. We may drop one pair of brackets and write $[\ell]^k$ instead of $[[\ell]]^k$.
Given a $k$-graph $H$ and a subset $S\in [V(H)]^s$, the degree of $S$, denoted by $\deg_H(S)$ or $\deg(S)$, is the number of edges of $H$ containing $S$ as a subset. Let $N_H(S)$ or $N(S)$ denote the neighbor set of $S$, i.e., $N_H(S)=\{ S' \in [V(H)]^{k-s} : S'\cup S\in E(H)\}$.
The minimum $s$-degree $\delta_s(H)$ is the minimum of $\deg(S)$ over all $s$-subsets $S$ of $V(H)$.

\begin{thm}\label{k-partite}
For $k \geq 3$, a $k$-partite $k$-graph $F\in\factor_k$ if and only if there exists $v^*\in V(F)$ such that $|e\cap e'|\leq 1$ for any two edges $e,e'$ with $v^*\in e$ and $v^*\notin e'$.
\end{thm}

\subsubsection{A characterization of $\factor_3$.}

%Now we turn to a full characterization of $\factor_3$.
We now turn to give a complete answer to Problem~\ref{prob} when $k=3$, namely, a full characterization of $\factor_3$.
We start with a formal definition of Tur\'an densities in $(p,\mu)$-dense $k$-graphs, borrowed from~\cite{Reiher2018Hypergraphs}.
Given a $k$-graph $F$, we define:
\begin{equation*}
\begin{split}
\label{dot-turan-dense}
\pi_{\points}(F) = \sup \{ p\in [0,1] & : \text{for\ every\ } \mu>0 \ \text{and\ } n_0\in \mathbb{N},\ \text{there\ exists\ an\ } F \text{-free} \\
&\quad (n, p,\mu)\ k \text{-graph\ } H\ \text{with\ } n\geq n_0 \}.
\end{split}
\end{equation*}

By $(p,\mu)$-denseness, every $(n, p,\mu)$ $k$-graph $H$ has at most $(2\mu/p) n$ vertices of degree lower than $p^2 n^{k-1}$.
Indeed, let $X$ be the set of all such low degree vertices and we have
\[
(k-1)!|X|p^2 n^{k-1}\ge e(X, V(H),\dots, V(H)) \ge p |X| n^{k-1} - \mu n^k,
\]
then the claim follows.
Therefore, for any $F$ with $\pi_{\points}(F)>0$, the $F$-free $(n, p,\mu)$ $k$-graph certifying this fact can be modified to obtain an $(n', p,\mu', p^2)$ $k$-graph by the deletion of up to $(2\mu/p) n$ vertices.
This implies that in the definition of $\pi_{\points}(F)$ above it is equivalent to consider the collection of all $(n, p,\mu,\alpha)$ $k$-graphs.
In particular, we infer that all $F\in \factor_k$ (or $\cover_k$) must satisfy $\pi_{\points}(F)=0$.
The following celebrated characterization of $3$-graphs $F$ with $\pi_{\points}(F)=0$ was recently obtained by Reiher, R\"odl and Schacht~\cite{Reiher2018Hypergraphs}.
For a $k$-graph $F$, let $\partial F:=\{S\in [V(F)]^{k-1}: \exists E\in E(F), S\subseteq E\}$ be the \emph{shadow} of $F$.

\begin{thm}\cite{Reiher2018Hypergraphs}\label{0}
For a $3$-graph $F$, the following are equivalent:
\begin{enumerate}[label=$(\alph*)$]
  \item[{\rm (a)}] $\pi_{\points}(F)=0$;
  \item[{\rm (b)}] there is an enumeration of the vertex set $V(F)=\{v_1,v_2,\dots,v_f\}$ and there is a $3$-coloring $\varphi:\partial F \longmapsto \{{\color {red}red},~{\color{blue}blue},~{\color{green}green}\}$ of the pairs of vertices covered by hyperedges of $F$ such that every hyperedge $\{v_i,v_j,v_k\}\in E(F)$ with $i<j<k$ satisfies: \label{item:b}
\end{enumerate}
\[
\varphi(v_i,v_j)={\color {red}red},\  \ \varphi(v_i,v_k)={\color{blue}blue},\ \ \varphi(v_j,v_k)={\color{green}green}.
\]
\end{thm}

Using Theorem~\ref{k-graph} and some ideas in the proof of Theorem~\ref{0} from~\cite{Reiher2018Hypergraphs}, we are able to give an explicit characterization of the 3-graphs in $\factor_3$.
%Furthermore, our next result characterises all $3$-graphs with factorable property.
%Before stating this theorem, we say a vertex $v$ in $V(F)$ is special if $N_F(v)\cap N_F(u)=\emptyset$ for all $u\in V(F)\setminus\{v\}$.
%If $v$ is a vertex in a $3$-graph $H$, the \textit{link} $L_x$ of $v$ is the graph with vertex set $V(H)\setminus \{v\}$ and edges those pairs who form an edge with x.
\begin{thm}\label{equal}
A $3$-graph $F\in \factor_3$ if and only if it satisfies the following properties.
\begin{enumerate}[label=$(\roman*)$]
  \item $\pi_{\points}(F)=0$;\label{item:i}
  \item there is a vertex $v^*\in V(F)$ and a partition $\mathcal{P}=\{X,Y, \{v^*\}\}$ of $V(F)$ such that $N(v^*)\subseteq X\times Y$ and for any $x\in X$ and $y\in Y$, $N(x)$, $N(y)$ and $N(v^*)$ are pairwise disjoint. \label{item:ii}
\end{enumerate}
\end{thm}

Theorem~\ref{equal} says that $F\in \factor_3$ must satisfy both \rm (ii) and Theorem~\ref{0}~\rm(b).
In particular, elements of $N(v^*)$ cannot belong to any other edges of $F$.
In Section $4$ we show that the two conditions are in fact ``compatible'', namely, we can take an ordering of $V(F)$ starting with $v^{*}$, followed by an enumeration of elements of $X$, and then an enumeration of elements of $Y$ which overall satisfies~\rm(b).

%We conclude this section by emphasizing again the significance of our Theorem~\ref{k-graph}, which reduces the factor problem to a ``rooted embedding'' problem, which greatly reduces the technicality in the proofs.
%For the proof of Theorem~\ref{k-partite}, we only need to prove a ``rooted embedding'' problem of $k$-partite $k$-graphs via Theorem~\ref{k-graph}.
%Our proof of Theorem~\ref{equal} then builds on the approach of that of Theorem~\ref{0} via the regularity method, with additional ideas incorporating property~\rm(ii).
%We believe our approach will be useful in solving Problem~\ref{prob} for $k\ge 4$.
%In particular, this would imply a generalization of Theorem~\ref{equal}, which we believe to be quite interesting, as it will contain a generalization of Theorem~\ref{0} to $k$-graphs.
%For the proof of Theorem~\ref{k-partite}, we only need to prove a ``rooted embedding'' problem of $k$-partite $k$-graphs via the Theorem~\ref{k-graph}.

\subsection{Proof Techniques}

%In this section we briefly discuss our proof techniques.
The necessities in Theorems~\ref{k-partite} and~\ref{equal} come from probabilistic constructions using random colorings of $E(K_n)$ together with small modifications, which are inspired by similar constructions along this line of research.
In particular, similar constructions were given in recent works~\cite{arajo2020, HSW21}.

The main tools for the other direction are the \emph{hypergraph regularity method} and the \emph{lattice-based absorption method}.

The hypergraph regularity method, as an extension of the Szemer\'edi's regularity lemma for graphs, has been a celebrated tool for embedding problems in hypergraphs.
%It has found a number of exciting applications in additive combinatorics, extremal combinatorics and theoretic computer science, e.g., the hypergraph removal lemmas and subgraph testings.
%It is known that such a regularity partition for $k$-graphs can be efficiently established, namely, in time $O(n^{3k})$.
We use a popular version of the regularity lemma for $k$-graphs due to R\"odl and Schacht~\cite{regular-lemmas2007}, with an associated counting lemma~\cite{CFKO} and restriction lemma~\cite{K2010Hamilton}.
For $k=3$, we also use a recent version due to Reiher, R\"odl and Schacht~\cite{Reiher_2018}, with an associated embedding lemma~\cite{Reiher_2018}.
More precisely, we use the hypergraph regularity method to prove the $F$-cover property, namely, to find a copy of $F$ covering any given vertex in the host $k$-graph.
Further, we actually have \emph{supersaturation}, namely, we can find $\Omega(n^{f-1})$ such copies of $F$.
However, unlike standard proofs of supersaturation results in the dense setting using hypergeometric distribution, where most constant-sized subsets inherit the density condition or the minimum degree condition, it is not clear to us whether most constant-sized subsets inherit the $(p,\mu)$-denseness.
Instead, our proof of the supersaturation result (Lemma~\ref{supersaturation}) uses the regularity method, which is in the same vein as in classical proofs of graph removal lemmas.
%when he studied the decision problem for perfect matchings in dense $k$-graphs, where he used the method to establish a sharp complexity result (and a further refinement was recently obtained in~\cite{HaKe}).
%Applying it to the $F$-factor problem, the key feature of this method implies a polynomial-time deterministic algorithm that builds a partition of $V(H)$ which records rich information for the distribution of the copies of $F$ in $H$.

The lattice-based absorption method was developed by the second author in~\cite{Han2017Perfect, HaKe, Han_2020}, which has found various applications in embedding spanning substructures under quite a few different contexts.
The method was developed to (efficiently) deal with a series of divisibility constructions in the perfect matching and $F$-factor problems, namely, the \emph{divisibility barriers}, which are well-structured $k$-graphs and thus far from being random.
Given the ``random-like'' nature of our context, we use the lattice-based absorption method with a relatively light amount of work -- we need to prove a list of auxiliary results (Lemmas~\ref{vertex-reachable} --~\ref{S-closed}) but they are mostly standard.

For the proof of Theorem~\ref{k-partite}, we only need to prove a ``rooted embedding'' problem of $k$-partite $k$-graphs via Theorem~\ref{k-graph}.
Our proof of Theorem~\ref{equal} then builds on the approach of that of Theorem~\ref{0} via the regularity method, with additional ideas to incorporate property~\rm(ii).
We believe our approach will be useful in solving Problem~\ref{prob} for $k\ge 4$.
In particular, this would imply a generalization of Theorem~\ref{equal}, which we believe to be quite interesting, as it probably will yield a generalization of Theorem~\ref{0} to $k$-graphs.

%\medskip
%\noindent\textbf{Organization.}
\subsection*{Organization}
The rest of this paper is organized as follows.
In the next section, we prove a number of auxiliary results and use them to prove Theorem~\ref{k-graph}.
% while postpone a lemma to Section 3 as we shall use the hypergraph regularity method.
We review the hypergraph regularity method in Section 3 and then prove Lemma~\ref{supersaturation} and the backward implication of Theorem~\ref{k-partite}.
Section 4 contains a proof of the backward implication of Theorem~\ref{equal}.
Finally, in Section 5, we present a construction that shows the forward implications of Theorem~\ref{equal} and Theorem~\ref{k-partite}. The last section contains some remarks and open problems.

\section{Proof of Theorem \ref{k-graph}}

In this section, we shall prove Theorem~\ref{k-graph}, namely, for any $F\in \cover_k$, we prove that every $k$-graph $H$ satisfying the condition of Problem~\ref{prob} has an $F$-factor.
We follow the absorption approach which splits the proof of Theorem~\ref{k-graph} into the following two lemmas: one is on finding an almost perfect $F$-tiling in $H$, and the other is on ``finishing up" the perfect $F$-tiling by absorption.

\begin{lemma}[Almost Perfect Tiling Lemma]\label{Almost-factor}
Given $0< p,\alpha <1$ and a $k$-graph $F$ satisfying $\pi_{\points}(F)=0$, for any $0 <\omega < 1$, there exists an $n_0$ and $\mu > 0$ such that the following holds. Let $H$ be an $(n,p,\mu)$ $k$-graph with $n\geq n_0$.
%Then there exists $C \subseteq V (H)$ such that $|C|\leq \omega n$, $|C|\in f\mathbb{Z}$ and $H[\bar{C}]$ has an $F$-factor.
Then there exists an $F$-tiling that covers all but at most $\omega n$ vertices of $H$.
\end{lemma}

\begin{proof}
Given $\omega\in (0,1)$, choose $\mu$ small enough and $n$ large enough.
For any induced subgraph $H'$ of $H$ on $\omega n$ vertices, by Definition~\ref{def.},
$H'$ is $(p,\mu/\omega^k)$-dense, because the error term for edge counting in~\eqref{eq:count} is $\mu n^k = (\mu/\omega^k) |V(H')|^k$.
Therefore, we can greedily select vertex-disjoint copies of $F$ from $H$ until $\omega n$ vertices are left.
%Given $\omega\in (0,1)$, choose $\mu$ small enough and $n$ large enough.
%Since $F\in \cover_k$, as discussed above, it holds that $\pi_{\points}(F)=0$.
%Observe that any induced subgraph $H'$ of $H$ on $\omega n$ vertices is $(p,\mu/\omega^k)$-dense, because the error term for edge counting in~\eqref{eq:count} is $\mu n^k = (\mu/\omega^k) |V(H')|^k$.
%Therefore, we can greedily select vertex-disjoint copies of $F$ from $H$ until $\omega n$ vertices are left.
%Let $F_1,\dots,F_t$ be vertex-disjoint copies of $F$ and $C:=V(H)\setminus V(F_1)\setminus \cdots \setminus V(F_t)$ has no copy of $F$. Clearly, $|C|\in f\mathbb{Z}$. If $|C|> \omega n$. Let $H':=H[C]$ and notice that $H'$ is $(|C|,p,\mu/{\omega^k})$ $k$-graph. Lemma \ref{turan-dense=0} implies that the $H'$ contains a copy of $F$. This contradicts the selection of $C$. Thus, $|C|\leq \omega n$.
\end{proof}

\begin{lemma}[Absorbing Lemma]\label{Absorbing-Lemma}
Given $0<p,\alpha<1$ and $F\in \cover_k$ with $f:=v(F)$, for any $0 <\varepsilon < 1$, there exists an $n_0$ and $\mu,\omega > 0$ such that the following holds. Let $H$ be an $(n,p,\mu,\alpha)$ $k$-graph with $n\geq n_0$. Then there exists $W \subseteq V (H)$ with $|W| \leq \varepsilon n$ such that for any vertex set $U\subseteq V(H)\setminus W$ with $|U|\leq \omega n$ and $|U|\in f\mathbb{N}$, both $H[W]$ and $H[W \cup U]$ contain $F$-factors.
\end{lemma}

To prove Lemma ~\ref{Absorbing-Lemma}, our main tool is the lattice-based absorbing method developed recently by Han~\cite{Han2017Perfect}, which builds on the absorbing method initiated by R\"{o}dl, Ruci\'{n}ski and Szemer\'{e}di~\cite{R2015A}. We postpone the proof of Lemma ~\ref{Absorbing-Lemma} to the end of this section.
% since we need to use a series of other lemmas in the proof.

\begin{proof}[Proof of Theorem \ref{k-graph}]
Let $0<p,\alpha<1$ and $F\in \cover_k$ with $f:=v(F)$.  Let $n$ be sufficiently large and $\mu >0$ be sufficiently small. Choose $\varepsilon>0$ small enough. First, select $\omega>0$ by Lemma~\ref{Absorbing-Lemma} and then $\mu_1>0$ by Lemma~\ref{Almost-factor}.
%Also, make $n_0$ large enough so that both Lemma \ref{Almost-factor} and \ref{Absorbing-Lemma} can be applied.
Let $H$ be an $(n,p,\mu,\alpha)$ $k$-graph with $n\geq n_0$ and $f\mid n$.
By Lemma~\ref{Absorbing-Lemma}, we have a $W \subseteq V (H)$ with $|W| \leq \varepsilon n$ such that for any vertex set $U\subseteq V(H)\setminus W$ with $|U|\leq \omega n$ and $|U|\in f\mathbb{Z}$, both $H[W]$ and $H[W \cup U]$ contain $F$-factors.
Let $H':=H[V(H)\setminus W]$ and note that $H'$ is an $(n-|W|,p,\mu_1)$ $k$-graph.
By Lemma \ref{Almost-factor}, there is an $F$-tiling that leaves a set $C \subseteq V(H')$ of at most $\omega n$ vertices uncovered.
%, $|C|\in f\mathbb{Z}$ and $H'[\bar{C}]$ has an $F$-factor.
Since $n, |W|\in f\mathbb{N}$, it holds that $|C|\in f\mathbb{N}$.
By the absorbing property of $W$, there exists an $F$-factor on $H[W\cup C]$. Thus we get an $F$-factor of $H$.
%Now Lemma \ref{Absorbing-Lemma} implies that $H[W\cup C]$ contains an $F$-factor. The $F$-factor of $H[W\cup C]$ and the $F$-factor of $H'[\bar{C}]$ together give an $F$-factor of $H$.
\end{proof}

Next, we state and prove several lemmas that are useful for the proof of Lemma \ref{Absorbing-Lemma}.
%Although we follow known approaches (the absorption method) in the literature, we still need to employ some new ideas.
The key property is the $F$-cover property and we first strengthen it as follows.

\begin{lemma}\label{supersaturation}
Given $F\in \cover_k$ with $f:=v(F)$, and $0<p,\alpha<1$, there exist $\mu,\eta>0$ and an $n_0\in \mathbb{N}$ such that the following holds.
If $H$ is an $(n,p,\mu,\alpha)$ $k$-graph with $n\geq n_0$, then for any vertex $w$ in $V(H)$, $w$ is contained in at least $\eta n^{f-1}$ copies of $F$.
\end{lemma}

In the usual setting (e.g., problems with minimum degree-type conditions), Lemma~\ref{supersaturation} might be proved using the \emph{supersaturation trick} via \emph{hypergeometric distribution}.
To be more precise, most constant-sized sets in the host graph actually inherit the degree assumptions, and thus we can find an embedding within each such set and then conclude the result via counting.
However, it is not clear for us whether constant-sized induced subgraphs of $(p,\mu)$-dense $k$-graphs are still quasi-random.
In contrast, we use the regularity method to prove Lemma~\ref{supersaturation} and therefore postpone its proof to Section 3.
To prove Lemma \ref{Absorbing-Lemma}, we  use the notion of reachability introduced by Lo and Markstr\"{o}m~\cite{Lo2015}.
Given a $k$-graph $F$ of order $f$, $\beta> 0$ and $i \in \mathbb{N}$, two vertices $u, v$ in a $k$-graph $H$ on $n$ vertices are \textit{$(F, \beta, i)$-reachable} (in $H$) if there are at least $\beta n^{if-1}$ $(if-1)$-sets $W$ such that both $H[\{u\}\cup W]$ and $H[\{v\}\cup W]$ contain $F$-factors. In this case, we call $W$ a \emph{reachable set} for $u$ and $v$. A vertex set $A$ is \emph{$(F, \beta, i)$-closed} in $H$ if every two vertices in $A$ are $(F, \beta, i)$-reachable in $H$. For $x \in V (H)$, let $\tilde{N}_{F, \beta, i}(x)$ be the set of vertices that are $(F, \beta, i)$-reachable to $x$.

The following lemmas explore the reachability properties of $(p,\mu)$-dense $k$-graphs.
We first show that two vertices are reachable if they have large common neighborhood.

%%%%%%%%%%%%%%%%%%%%%%%%%%%%%%%%%%%%%%%%%%%%%%%%%%%%%%%%%%%%%%%%%%%%%%%%%%%%%%%%%%%%%%%%%%%
\begin{lemma}\label{vertex-reachable}
Given $0<p,\alpha<1$, $\alpha' \ll \alpha $ and $F\in \cover_k$ with $f:=v(F)$, there exists an $n_0\in \mathbb{N}$ and $0<\mu,\beta <1$ such that the following holds.
Let $H$ be an $(n,p,\mu,\alpha)$ $k$-graph with $n\geq n_0$. If two vertices $u,v$ satisfy $|N(u)\cap N(v)|\geq \alpha' n^{k-1}$, then $u,v$ are $(F, \beta , 1)$-reachable in $H$.
\end{lemma}
\begin{proof}
Construct an auxiliary hypergraph $H'$ from $H$ by deleting vertex $v$ and an edge set $E'=\{ S'\cup \{u\}\in E(H): S'\in N_H(u)\setminus N_H(v)\}$.
%We claim that there exist constants $\mu ',\alpha '$ such that $H'$ is an $(n-1,p,\mu ',\alpha ')$ $k$-graph.
We claim that $H'$ is an $(n-1,p,2\mu,\alpha')$ $k$-graph.
Indeed, since $H$ is $(p,\mu)$-dense and $n\geq n_0$,
%we can choose $\mu '$ such that for all $X_1,\dots,X_k\subseteq V(H')$,
%\[
%\begin{split}
%e_{H'}(X_1,\dots,X_k) & \ge e_{H}(X_1,\dots,X_k)-k! \binom{n}{k-1} \\
%& \ge p|X_1|\cdots|X_k|-\mu n^k-k! \binom{n}{k-1}\ge p|X_1|\cdots|X_k|-\mu ' (n-1)^k.
%\end{split}
%\]
the edge counting~\eqref{eq:count} has error term $\mu n^k + 2 n^{k-1} \le 2\mu (n-1)^k$, namely, $H'$ is $(p,2\mu)$-dense.
%
%On the other hand, for every vertex $w\in V(H')$, we define $d_{H'}(w)$ to be the degree of $w$ in $H'$.
For the minimum degree of $H'$, we have $\deg_{H'}(u)=|N_{H}(v)\cap N_{H}(u)|\geq \alpha' n^{k-1}$; for other vertices $w\in V(H')\setminus \{ v\} $, as we delete $v$ from $H$, we have $\deg_{H'}(w)\ge \alpha {n}^{k-1}-\binom{n}{k-2}\ge \alpha' n^{k-1}$.
%Since $\alpha ' \ll \alpha ,\varepsilon $, for every vertex $w\in V(H')$$$d_{H'}(w)\ge \alpha '\binom{n-1}{k-1}.$$
Thus, $H'$ is an $(n-1,p,2\mu,\alpha')$ $k$-graph.

Since $\mu$ is chosen to be small and $n_0$ is chosen to be large, by Lemma~\ref{supersaturation} there exists $\beta$ such that $u$ is contained in at least $\beta n^{f-1}$ copies of $F$ in $H'$.
By the definition of $H'$, in those copies of $F$ we may replace $u$ by $v$, namely, there are $\beta n^{f-1}$ $(f-1)$-sets $W$ such that both $H[\{u\}\cup W]$ and $H[\{v\}\cup W]$ contain a copy of $F$. Therefore, $u,v$ are $(F, \beta, 1)$-reachable in $H$.
\end{proof}

Using the previous result, we show that among a constant number of vertices, there are two of them that are reachable.

%%%%%%%%%%%%%%%%%%%%%%%%%%%%%%%%%%%%%%%%%%%%%%%%%%%%%%%%%%%%%%%%%%%%%%%%%%%%%%%%%%%%%%%%%%%
\begin{lemma}\label{c+1}
Given $0<p,\alpha<1$ and $F\in \cover_k$ with $f:=v(F)$,
there exists an $n_0\in \mathbb{N}$ and $0<\mu,\beta <1$ such that the following holds.
Let $H$ be an $(n,p,\mu,\alpha)$ $k$-graph with $n\ge n_0$.
Then every set of $\lfloor 1/\alpha \rfloor+1$ vertices in $V(H)$ contains two vertices that are $(F, \beta , 1)$-reachable in $H$.
\end{lemma}

\begin{proof}
Let $c := \lfloor 1/\alpha \rfloor$ and choose $\alpha'$ such that $(c+1)\alpha > 1+(c+1)^2\alpha'$.
Let $\beta ,\mu >0$ be given by Lemma~\ref{vertex-reachable}.
Since $H$ is an $(n,p,\mu,\alpha)$ $k$-graph, we have $\delta_1 (H)\ge \alpha {n}^{k-1}$.
%Note that by Lemma~\ref{vertex-reachable}, distinct $u,v$ $\in V(H)$ with $|N(v)\cap N(u)|\geq \varepsilon n^{k-1}$ are $(F, \beta, 1)$-reachable in $H$ for some $\beta >0$.
%Indeed, the degree sum of any $c+1$ vertices is at least $( c +1)\alpha \binom{n}{k-1}$. Since $\varepsilon \ll \alpha $,
Given any $c+1$ vertices, since their degree sum is at least
\[
(c+1)\alpha {n}^{k-1} > (1+ (c+1)^{2} \alpha'){n}^{k-1},
\]
there exist two vertices $u,v$ such that $|N(u)\cap N(v)|\geq \alpha' n^{k-1}$.
By Lemma~\ref{vertex-reachable}, they are $(F, \beta , 1)$-reachable.
\end{proof}

Note that our host $k$-graph may contain vertices with small ``reachable neighborhoods'' (if we view the reachability information as a graph).
The following lemma says that given the assumption that every constant-sized set of vertices contains two reachable vertices, then we can find a very large set $S$ in which all vertices have large reachable neighborhoods.

%%%%%%%%%%%%%%%%%%%%%%%%%%%%%%%%%%%%%%%%%%%%%%%%%%%%%%%%%%%%%%%%%%%%%%%%%%%%%%%%%%%%%%%%%%%
\begin{lemma}\label{S-closed}
Let $\delta ,\beta >0$ and integers $c,k,f\geq 2$ be given. Let $F$ be a $k$-graph on $f$ vertices. Assume that $H$ is a $k$-graph on $n$ vertices satisfying that every set of $c+1$ vertices contains two vertices that are $(F,\beta ,1)$-reachable in $H$. Then there exists $S\subseteq V(H)$ with $|S|\geq (1-c\delta )n$ such that $|\tilde{N}_{F, \beta , 1}(v)\cap S|\geq \delta n$ for any $v\in S$.
\end{lemma}
\begin{proof}
%To find the subset $S$ from $V(H)$, we try to delete vertex $v$ with $|\tilde{N}_{F, \beta, 1}(v,H)|<\delta n$ one by one, until only vertices of large $|\tilde{N}_{F, \beta, 1}(v,H)|$ remain.
Let $H$ be a $k$-graph on $n$ vertices satisfying the condition of Lemma~\ref{S-closed}. We greedily identify vertices with few ``reachable neighbors'' and remove the vertex together with the vertices reachable to it from $H$.
Set $V_0:=V(H)$. If there is a vertex $v_0\in V_0$ such that $|\tilde{N}_{F, \beta , 1}(v_0)\cap V_0|<\delta n$, then let $A_0:=\{v_0\}\cup \tilde{N}_{F, \beta , 1}(v_0)$ and let $V_1:=V_0\setminus A_0$.
Next, we check $V_1$ -- if there still exists a vertex $v_1\in V_1$ such that $|\tilde{N}_{F, \beta , 1}(v_1)\cap V_1|<\delta n$, then let $A_1:=\{v_1\}\cup \tilde{N}_{F, \beta , 1}(v_1)$ and let $V_2:=V_1\setminus A_1$ and repeat the procedure until no such $v_j$ exists.
%If $A_j$ has been defined, then continue to check $V_{j+1}$.
Suppose we stop and obtain a set of vertices $v_0,\dots, v_s$.
Observe that every two vertices of $v_0,\dots, v_s$ are not $(F, \beta , 1)$-reachable in $H$, which implies $s< c$ and $|\bigcup _{0\leq i\leq s}A_i|\leq c\delta n$.
%We choose $0<\eta<\lfloor 1/\alpha \rfloor \delta$ and s
Set $S=V(H)\setminus \bigcup _{0\leq i\leq s}A_i $ and thus $|\tilde{N}_{F, \beta , 1}(v)\cap S|\geq \delta n$ for any $v\in S$.
\end{proof}

The following lemma was proved in~\cite {Han_2020}, which can be used to find a partition of the subset $S$ in Lemma~\ref{S-closed} that possesses a rich property.

\begin{lemma}{\rm{(}\cite[Theorem 6.3]{Han_2020}\rm{)}}.
\label{partation}
Given $\delta>0$, integers $c,k,f\geq2$ and $0 <\beta \ll 1/c, \delta, 1/f$, there exists a constant $\gamma >0$ such that the following holds for all sufficiently large $n$. Let $F$ be a $k$-graph on $f$ vertices. Assume $H$ is a $k$-graph on $n$ vertices and $S\subseteq V(H)$ is such that $|\tilde{N}_{F, \beta , 1}(v)\cap S|\geq \delta n$ for any $v\in S$. Further, suppose every set of $c+1$ vertices in $S$ contains two vertices that are $(F, \beta , 1)$-reachable in $H$. Then we can find a partition $\mathcal{P}$ of $S$ into $V_1,\dots,V_r$ with $r \leq \min\{c,\lfloor1/{\delta}\rfloor\}$ such that for any $i \in [r]$, $|V_i|\geq (\delta -\beta )n$ and $V_i$ is $(F, \gamma , 2^{c-1})$-closed in $H$. \qed
\end{lemma}

%%%%%%%%%%%%%%%%%%%%%%%%%%%%%%%%%%%%%%%%%%%%%%%%%%%%%%%%%%%%%%%%%%%%%%%%%%%%%%%%%%%%%%%%%%%%
Next, to show that the subset $S$ in Lemma~\ref{S-closed} is closed,
we use the following definitions introduced by Keevash and Mycroft \cite{Ke2015}. Let $r\geq f\geq k\geq 2$ be integers and let $F$ be a $k$-graph on $f$ vertices. Suppose that $H$ is a $k$-graph with a partition $\mathcal{P} =\{V_0,V_1,\dots,V_r\}$ of $V (H)$. The \textit{index vector} $\mathbf{i}_{\mathcal{P}}(S)\in \mathbb{Z}^r$ of a subset $S\subseteq V(H)$ with respect to $\mathcal{P}$ is the vector whose coordinates are the sizes of the intersections of $S$ with $V_1,\dots,V_r$. We call a vector $\mathbf{i}\in \mathbb{Z}^r$ an \emph{$s$-vector} if all its coordinates are nonnegative and their sum is $s$. Given $\lambda>0$, an $f$-vector $\mathbf{v} \in \mathbb{Z}^r$ is called a \emph{$\lambda$-robust $F$-vector} if at least $\lambda |V(H)|^f$ copies $F'$ of $F$ in $H$ satisfy $\mathbf{i}_{\mathcal{P}}(V(F'))=\mathbf{v}$. Let $I^{\lambda}_{\mathcal{P},F}(H)$ be the set of all $\lambda$-robust $F$-vectors and $L^{\lambda}_{\mathcal{P},F}(H)$ be the lattice generated by the vectors of $I^{\lambda}_{\mathcal{P},F}(H)$. For $j \in [r]$, let $\mathbf{u}_j\in \mathbb{Z}^r$ be the $j^{th}$ unit vector, namely, $\mathbf{u}_j$ has $1$ on the $j^{th}$ coordinate and $0$ on other coordinates. A \textit{transferral} is a non-zero difference $\mathbf{u}_j-\mathbf{u}_{\ell}$ of $\mathbf{u}_j$ and $\mathbf{u}_{\ell}$ for all $1\leq j<{\ell}\leq r$.

Given a partition $\mathcal{P} =\{V_0,V_1,\dots,V_r\}$ of $V (H)$, the following lemma from~\cite{Han2017Minimum} shows that $V\setminus V_0$ is closed
if $\mathbf{u}_j-\mathbf{u}_{\ell}\in L^{\lambda}_{\mathcal{P},F}(H)$ for all $1\leq j<{\ell}\leq r$.

\begin{lemma}{\rm{(}\cite[\rm Theorem 3.9]{Han2017Minimum}\rm{)}.}
\label{V-closed}
Let $i_0, r, k,f> 0$ be integers and let $F$ be a $k$-graph of order $f$. Given constants $\varepsilon, \gamma , \lambda > 0$, there exists $\gamma'>0$ and integers $i'_0,n_0$ such that the following holds for all $n\geq n_0$. Let $H$ be a $k$-graph on $n$ vertices with a partition $\mathcal{P} =\{V_0,V_1,\dots,V_r\}$ of $V (H)$ such that for each $j\in [r]$, $|V_j|\geq \varepsilon^2n$ and $V_j$ is $(F,\gamma , i_0)$-closed in $H$. If $\mathbf{u}_j -\mathbf{u}_{\ell} \in  L^{\lambda}_{\mathcal{P},F}(H)$ for all $1 \leq  j <{\ell} \leq r$, then $V(H)\setminus V_0$ is $(F, \gamma ', i'_0)$-closed in $H$. \qed
\end{lemma}

%%%%%%%%%%%%%%%%%%%%%%%%%%%%%%%%%%%%%%%%%%%%%%%%%%%%%%%%%%%%%%%%%%%%%%%%%%%%%%%%%%
Our proof of Lemma~\ref{Absorbing-Lemma} is based on the following lemma.
%Let $f\geq k\geq 2$ be integers and $F$ be a $k$-graph on $f$ vertices.
%We call an $m$-set $A$ an \emph{absorbing $m$-set} for an $f$-set $S$ if $A\cap S=\emptyset$ and both $H[A]$ and $H[A \cup S]$ contain $F$-factors.
%Denote by $\mathcal{A}^m(S)$ the set of all absorbing $m$-sets for $S$.

\begin{lemma} \label{absorb}
Given $0 <\eta , \gamma ' <1$, and $i'_0, k,f\in \mathbb{N}$, let $F$ be a $k$-graph with $f$ vertices.
There exists $\omega>0$ such that the following holds for all sufficiently large integers $n$.
Suppose $H$ is a $k$-graph on $n$ vertices with the following two properties.
\begin{enumerate}[label=$(\roman*)$]
  \item[{\rm (i)}] For any $v\in V(H)$, there are at least $\eta n^{f-1}$ copies of $F$ containing it;
  \item[{\rm (ii)}]  there exists $V_0\subseteq V$ with $|V_0|\leq \eta^2n$ such that $V (H)\setminus V_0$ is $(F, \gamma ', i'_0)$-closed in $H$.
\end{enumerate}
Then there exists a vertex set $W$ with $V_0\subseteq W\subseteq V(H)$ and $|W|\leq  \eta n$ such that for any vertex set $U \subseteq V (H)\setminus W$ with $|U| \leq \omega n$ and $|U| \in f\mathbb{N}$, both $H[W]$ and $H[U \cup W]$ contain $F$-factors. \qed
\end{lemma}

Han--Zang--Zhao~\cite[Lemma 3.6]{Han2017Minimum} proved Lemma~\ref{absorb} for $F$ being a complete $3$-partite $3$-graph.
Their proof indeed works for any $k$-graph $F$ and thus we choose not to reproduce the proof here.
Finally we put everything together and prove Lemma~\ref{Absorbing-Lemma}.

\begin{proof}[Proof of Lemma~\ref{Absorbing-Lemma}]
Let $F\in \cover_k$ with $f:=v(F)$ and $0<p,\alpha,\varepsilon<1$.
%Given $0<p,\alpha<1$, let $H$ be an $(n,p,\mu,\alpha)$ $k$-graph such that any vertex $v\in V(H)$ is contained in at least $\varepsilon n^{f-1}$ copies of $F$ for some $\varepsilon>0$.
We first choose $0<\eta<\varepsilon$ by Lemma~\ref{supersaturation} with $p$ and $\min\{\alpha, p/k!\}$ in place of $\alpha$.
For other parameters, we select $0<\beta\ll\delta\ll\eta,\alpha,1/f$ as follows.
We apply Lemma \ref{c+1} and obtain $\beta$. Set $c := \lfloor 1/\alpha \rfloor$.
%and then apply Lemma \ref{S-closed} with $\beta'$ and $\delta$ and obtain $0<\eta\ll\varepsilon$.
Next we apply Lemma \ref{partation} with $\delta,c,\beta$ and obtain $r,\gamma>0$. Then we apply Lemma~\ref{V-closed} with $\lambda = (\eta/f)(\delta/2)^{f}$ and obtain $\gamma'>0$ and $i'_0\in \mathbb N$.
Finally, let $\omega$ be given by Lemma~\ref{absorb} with inputs $\eta$, $\gamma'$ and $i'_0$.
Let $\mu>0$ be sufficiently small and $n\in f\mathbb{N}$ be sufficiently large.

Let $H$ be an $(n,p,\mu,\alpha)$ $k$-graph.
By Lemma~\ref{supersaturation}, every vertex $v\in V(H)$ is contained in at least $\eta n^{f-1}$ copies of $F$.
By Lemma \ref{c+1}, every set of $c+1$ vertices in $V(H)$ contains two vertices that are $(F,\beta,1)$-reachable.
By Lemma \ref{S-closed}, there exists $S\subseteq V(H)$ such that $|V(H)\setminus S|\le c\delta n$ and $|\widetilde{N}_{F, \beta, 1}(v)\cap S|\geq\delta n$ for every $v\in S$.
Therefore, following Lemma \ref{partation}, we obtain a partition $\mathcal{P}=\{V_0,V_1,V_2,\ldots,V_r\}$ of $V(H)$ with $r\leq\min\{c,1/\delta\}$ such that for any $i \in [r]$, $|V_i|\geq (\delta -\beta)n$ and $V_i$ is $(F, \gamma, 2^{c-1})$-closed in $H$.
Next we show that $\mathbf{u}_j -\mathbf{u}_{\ell} \in  L^{\lambda}_{\mathcal{P},F}(H)$ for all $1 \leq  j <{\ell} \leq r$.
Note that this will conclude that $V(H)\setminus V_0$ is $(F,\gamma',i'_0)$-closed by Lemma \ref{V-closed} and then the desired absorption set exists by Lemma \ref{absorb}, finishing the proof.

Without loss of generality, we may assume that $j=1$ and ${\ell}=2$.
%For some $0<\alpha_2\ll \alpha_1$,
For $i\in\{1,2\}$, let
\[X_i=\left\{v\in V_i~:~{\rm deg}_{V_1}(v)<\frac{p}{2(k-1)!}|V_1|^{k-1}\right\}
\]
and $V'_i=V_i\setminus X_i$. Since $H$ is $(p,\mu)$-dense, we obtain
\[
p|X_i| |V_1|^{k-1}-\mu n^k\leq e_H(X_i,V_1,\ldots,V_1)\leq \frac{p}{2}|X_i||V_1|^{k-1},
\]
which implies that $|X_i|\leq\frac{\mu n^k}{(p/2)|V_1|^{k-1}}<\sqrt{\mu}|V_1|$. Therefore, for any $v\in V'_i$ it holds that
\[{\rm deg}_{V'_1}(v)\geq\frac{p}{2(k-1)!}|V_1|^{k-1}-\sqrt{\mu}|V_1|^{k-1}>\frac{p}{k!}|V_1|^{k-1}.\]
Consequently, $H[V'_1]$ is a $(|V'_1|,p,\frac{2^k\mu}{\delta^k},\frac{p}{k!})$ $k$-graph and $H[V'_1\cup\{v\}]$ is a $(|V'_1|+1,p,\frac{2^k\mu}{\delta^k},\frac{p}{k!})$ $k$-graph for any $v\in V'_2$.
Since $F\in \cover_k$, it follows that every vertex $v\in V'_1~(\text{resp.}~ u\in V'_2)$ is contained in at least $\eta |V'_1|^{f-1}$ copies of $F$ in $H[V'_1]~(\text{resp. in}~ H[V'_1\cup\{u\}])$.
This implies $(\eta/f) |V'_1|^{f} \ge (\eta/f)(\delta/2)^{f} n^f$ copies of $F$ in $V'_1$,
and $|V'_2|\cdot \eta |V'_1|^{f-1}\ge \eta (\delta/2)^f n^f$ copies of $F$ with one vertex in $V'_2$ and $f-1$ vertices in $V'_1$.
By the choice of $\lambda$, we deduce that $\mathbf{v}_1,\mathbf{v}_2\in L^{\lambda}_{\mathcal{P},F}(H)$ with $\mathbf{v}_1=f\mathbf{u}_1$ and $\mathbf{v}_2=(f-1)\mathbf{u}_1+\mathbf{u}_2$.
Thus $\mathbf{u}_1 -\mathbf{u}_2=\mathbf{v}_1 -\mathbf{v}_2\in L^{\lambda}_{\mathcal{P},F}(H)$.
\end{proof}

\section{The Hypergraph Regularity Lemma}
In this section, we state the regularity lemma, the counting lemma and the restriction lemma for hypergraphs. Then we shall prove Lemma~\ref{supersaturation} and the backward implication of Theorem~\ref{k-partite}.
We follow the approach from R\"odl and Schacht~\cite{regular-lemmas2007,MR2351689}, combined with results from~\cite{CFKO} and~\cite{K2010Hamilton}. The central concepts of hypergraph regularity lemma are \emph{regular complex} and \emph{equitable partition}. Before we state the hypergraph regularity lemma, we introduce some necessary notation.
For reals $x,y,z$ we write $x = y \pm z$ to denote that $y-z \leq x \leq y+z$.

\subsection{Regular complexes}
A \emph{hypergraph} $\mathcal{H}$ consists of a vertex set $V(\mathcal{H})$ and an edge set $E(\mathcal{H})$, where every edge $e \in E(\mathcal{H})$ is a non-empty subset of $V(\mathcal{H})$. So a $k$-graph as defined earlier is a $k$-uniform hypergraph in which every edge has size $k$. A hypergraph $\mathcal{H}$ is a \emph{complex} if whenever $e \in E(\mathcal{H})$ and $e'$ is a non-empty subset of $e$ we have that $e'\in E(\mathcal{H})$. All the complexes considered in this paper have the property that all vertices are contained in an edge. A complex $\mathcal{H}^{\leq k}$ is a \emph{$k$-complex} if all the edges of $\mathcal{H}^{\leq k}$ consist of at most $k$ vertices. Given a $k$-complex $\mathcal{H}^{\leq k}$, for each $i \in [k]$, the edges of size $i$ are called \emph{$i$-edges} of $\mathcal{H}^{\leq k}$ and we denote by $H^{(i)}$ the \emph{underlying $i$-graph} of $\mathcal{H}^{\leq k}$: the vertices of $H^{(i)}$ are those of $\mathcal{H}^{\leq k}$ and the edges of $H^{(i)}$ are the $i$-edges of $\mathcal{H}^{\leq k}$. Note that a $k$-graph $H$ can be turned into a $k$-complex by making every edge into a \emph{complete $i$-graph} $K^{(i)}_k$ (i.e., consisting of all $\binom{k}{i}$ different $i$-tuples on $k$ vertices), for each $ i\in [k]$.

Given positive integers $s\geq k$, an \emph{$(s,k)$-graph} $H^{(k)}_s$ is an $s$-partite $k$-graph, by which we mean that the vertex set of $H^{(k)}_s$ can be partitioned into sets $V_1,\dots, V_s$ such that every edge of $H^{(k)}_s$ meets each $V_i$ in at most one vertex for $i\in [s]$. Similarly, an \emph{$(s,k)$-complex} $\mathcal{H}^{\leq k}_s$ is an $s$-partite $k$-complex.

Given $i \geq 2$, let $H^{(i)}_i$ and $H^{(i-1)}_i$ be on the same vertex set. We denote by $\mathcal{K}_i(H^{(i-1)}_i)$ for the family of $i$-sets of vertices which form a copy of the complete $(i-1)$-graph $K^{(i-1)}_i$ in $H^{(i-1)}_i$. We define the \emph{density} of $H^{(i)}_i$ w.r.t. (with respect to) $H^{(i-1)}_i$ to be
\[
d(H^{(i)}_i|H^{(i-1)}_i):= \begin{cases}
\frac{|E(H^{(i)}_i)\cap \mathcal{K}_i(H^{(i-1)}_i)|}{|\mathcal{K}_i(H^{(i-1)}_i)|} &\text{if\ } |\mathcal{K}_i(H^{(i-1)}_i)|>0,\\
0&\text{otherwise}.
\end{cases}
\]
More generally, if $\mathbf{Q}:= (Q(1), Q(2),\dots, Q(r))$ is a collection of $r$ subhypergraphs of $H^{(i-1)}_i$, we define $\mathcal{K}_i(\mathbf{Q}):= \bigcup^r_{j=1}\mathcal{K}_i(Q(j))$ and
\[
d(H^{(i)}_i|\mathbf{Q}):= \begin{cases}
\frac{|E(H^{(i)}_i)\cap \mathcal{K}_i(\mathbf{Q})|}{|\mathcal{K}_i(\mathbf{Q})|} &\text{if\ } |\mathcal{K}_i(\mathbf{Q})|>0,\\
0&\text{otherwise}.
\end{cases}
\]

We say that an $H^{(i)}_i$ is \emph{$(d_i, \delta, r)$-regular} w.r.t. an $H^{(i-1)}_i$ if every $r$-tuple $\mathbf{Q}$ with $|\mathcal{K}_i(\mathbf{Q})| \geq \delta|\mathcal{K}_i(H^{(i-1)}_i)|$ satisfies $d(H^{(i)}_i|\mathbf{Q}) = d_i \pm \delta$. Instead of $(d_i, \delta, 1)$-regular, we refer to $(d_i, \delta)$-\emph{regular}.
Moreover, for $s \geq i\geq 2$, we say that $H^{(i)}_{s}$ is \emph{$(d_i, \delta,r)$-regular} w.r.t. $H^{(i-1)}_{s}$ if for every $\Lambda_i\in  [s]^i$ the restriction $H^{(i)}_{s}[\Lambda_i]=H^{(i)}_{s}[\cup_{\lambda\in \Lambda_i }V_{\lambda}]$ is \emph{$(d_i, \delta,r)$-regular} w.r.t. the restriction $H^{(i-1)}_{s}[\Lambda_i]=H^{(i-1)}_{s}[\cup_{\lambda\in \Lambda_i }V_{\lambda}]$.

\begin{defi}[$(d_2, \dots , d_k, \delta_k, \delta, r)$-regular complexes]
Given $3 \leq k \leq s$ and an $(s, k)$-complex $\mathcal{H}$, we say that $\mathcal{H}$ is \emph{$(d_2, \dots , d_k, \delta_k, \delta, r)$-regular} if the following conditions hold:\vspace{3mm}

$\bullet$ for every $i = 2, \dots , k-1$ and every $i$-tuple $\Lambda_i$ of vertex classes, either $H_s^{(i)}[\Lambda_i]$ is $(d_i, \delta)$-regular w.r.t $H_{s}^{(i-1)}[\Lambda_i]$ or $d(H_s^{(i)}[\Lambda_i]|H^{(i-1)}_s[\Lambda_i]) = 0$;

$\bullet$ for every $k$-tuple $\Lambda_k$ of vertex classes either $H^{(k)}_s[\Lambda_k]$ is $(d_k, \delta_k, r)$-regular w.r.t $H^{(k-1)}_s[\Lambda_k]$ or $d(H^{(k)}_s[\Lambda_k]|H^{(k-1)}_s[\Lambda_k]) = 0$.\vspace{3mm}
%Let $\delta>0$ and let $\mathbf{d}=(d_2,\dots, d_k)$ be a vector of non-negative reals. We say that a $(\ell, k)$-complex $\mathcal{H}^{\leq k}_{\ell}=\{H^{(i)}\}^k_{i=1}$ is $(\mathbf{d}, \delta_k,\delta,r)$-regular if $H^{(k)}$ is $(d_k, \delta_k,r)$-regular w.r.t. $H^{(k-1)}$ and if $H^{(i)}$ is $(d_i, \delta)$-regular w.r.t. $H^{(i-1)}$ for $i=2,\dots, k-1$.
\end{defi}

\subsection{Equitable partition}\label{section-equ-partition}
Suppose that $V$ is a finite set of vertices and $\mathcal{P}^{(1)}$ is a partition of $V$ into sets $V_1, \dots , V_{a_1}$, which will be called \emph{clusters}. Given $k \geq 3$ and any $j \in [k]$, we denote by $\mathrm{Cross}_j = \mathrm{Cross}_j (\mathcal{P}^{(1)})$, the set of all those $j$-subsets of $V$ that meet each $V_i$ in at most one vertex for $1\leq i \leq a_1$. For every subset $\Lambda \subseteq [a_1]$ with $2 \leq |\Lambda| \leq k-1$, we write $\mathrm{Cross}_{\Lambda}$ for all those $|\Lambda|$-subsets of $V$ that meet each $V_i$ with $i \in \Lambda$. Let $\mathcal{P}_{\Lambda}$ be a partition of $\mathrm{Cross}_{\Lambda}$. We refer to the partition classes of $\mathcal{P}_{\Lambda}$ as \emph{cells}. For each $i = 2, \dots , k-1$, let $\mathcal{P}^{(i)}$ be the union of all the $\mathcal{P}_{\Lambda}$ with $|\Lambda| = i$. So $\mathcal{P}^{(i)}$ is a partition of $\mathrm{Cross}_i$ into several $(i,i)$-graphs.

Set $1 \leq i \leq j$. For every $i$-set $I\in \mathrm{Cross}_i$, there exists a unique cell $P^{(i)}(I)\in \mathcal{P}^{(i)}$ so that $I\in P^{(i)}(I)$. We define for every $j$-set $J\in \mathrm{Cross}_j$ the \emph{polyad} of $J$ as:
\[
\hat{P}^{(i)}(J):=\bigcup\big\{P^{(i)}(I): I\in [J]^{i}\big\}.
\]
So we can view $\hat{P}^{(i)}(J)$ as a $(j,i)$-graph (whose vertex classes are clusters intersecting $J$). Let $\mathcal{\hat{P}}^{(j-1)}$ be the set consisting of all the $\hat{P}^{(j-1)}(J)$ for all $J\in \mathrm{Cross}_j$. It is easy to verify $\{\mathcal{K}_j(\hat{P}^{(j-1)}) : \hat{P}^{(j-1)}\in \mathcal{\hat{P}}^{(j-1)}\}$ is a partition of $\mathrm{Cross}_j$.

Given a vector of positive integers $\mathbf{a}=(a_1,\dots, a_{k-1})$, we say that $\mathcal{P}(k-1) = \{\mathcal{P}^{(1)}, \dots ,\mathcal{P}^{(k-1)}\}$ is a \emph{family of partitions} on $V$, if the following conditions hold:\vspace{3mm}

$\bullet$ $\mathcal{P}^{(1)}$ is a partition of $V$ into $a_1$ clusters.

$\bullet$ $\mathcal{P}^{(i)}$ is a partition of $\mathrm{Cross}_i$ satisfying $|\{P^{(i)}\in \mathcal{P}^{(i)}: P^{(i)}\subseteq \mathcal{K}_i(\hat{P}^{(i-1)})\}|=a_i$ for every $\hat{P}^{(i-1)}\in \mathcal{\hat{P}}^{(i-1)}$. Moreover for every $P^{(i)}\in \mathcal{P}^{(i)}$, there exists a $\hat{P}^{(i-1)}\in \mathcal{\hat{P}}^{(i-1)}$ such that $P^{(i)} \subseteq \mathcal{K}_i(\hat{P}^{(i-1)})$.

So for each $J \in \mathrm{Cross}_j$ we can view $\bigcup^{j-1}_{i=1}\hat{P}^{(i)}(J)$ as a $(j, j-1)$-complex.
\vspace{3mm}

\begin{defi}[$(\eta, \delta, t)$-equitable]\label{eq-partition}
Suppose $V$ is a set of $n$ vertices, $t$ is a positive integer and $\eta, \delta>0$. We say a family of partitions $\mathcal{P}=\mathcal{P}(k-1)$ is \emph{$(\eta, \delta, t)$-equitable} if it satisfies the following:
\begin{enumerate}
  \item $\mathcal{P}^{(1)}$ is a partition of $V$ into $a_1$ clusters of equal size, where $1/\eta \leq a_1 \leq t$ and $a_1$ divides $n$;
  \item for all $i = 2, \dots , k-1$, $\mathcal{P}^{(i)}$ is a partition of $\mathrm{Cross}_i$ into at most $t$ cells;
  \item there exists $\mathbf{d}= (d_2, \dots , d_{k-1})$ such that $d_i \geq 1/t$ and $1/d_i \in \mathbb{N}$ for all $i = 2, \dots , k-1$;
  \item for every $k$-set $K \in \mathrm{Cross}_k$, the $(k, k-1)$-complex $\bigcup^{k-1}_{i=1}\hat{P}^{(i)}(K)$ is $(\mathbf{d}, \delta,\delta,1)$-regular.
\end{enumerate}
\end{defi}
Note that the final condition implies that the cells of $\mathcal{P}^{(i)}$ have almost equal size for all $i = 2, \dots , k-1$.

\subsection{Statement of the regularity lemma.}
Let $\delta_k > 0$ and $r \in \mathbb{N}$. Suppose that $H$ is a $k$-graph on $V$ and $\mathcal{P} = \mathcal{P}(k-1)$ is a family of partitions on $V$. Given a polyad $\hat{P}^{(k-1)} \in \hat{\mathcal{P}}^{(k-1)}$, we say that $H$ is \emph{$(\delta_k, r)$-regular} w.r.t. $\hat{P}^{(k-1)}$ if $H$ is $(d_k, \delta_k, r)$-regular w.r.t. $\hat{P}^{(k-1)}$ for some $d_k$. Finally, we define that $H$ is \emph{$(\delta_k, r)$-regular} w.r.t. $\mathcal{P}$.

\begin{defi}[\emph{$(\delta_k, r)$-regular} w.r.t. $\mathcal{P}$]
We say that a $k$-graph $H$ is \emph{$(\delta_k, r)$-regular} w.r.t. $\mathcal{P} = \mathcal{P}(k-1)$ if
\[
\big|\bigcup\big\{\mathcal{K}_k(\hat{P}^{(k-1)}) : \hat{P}^{(k-1)}\in \mathcal{\hat{P}}^{(k-1)}\\
\text{and\ } H \text{\ is\ not\ } (\delta_k, r)\text{-regular\ w.r.t.\ } \hat{P}^{(k-1)} \big\}\big| \le \delta_k |V|^k.
\]
\end{defi}
This means that no more than a $\delta_k$-fraction of the $k$-subsets of $V$ form a $K_k^{(k-1)}$ that lies within a polyad with respect to which $H$ is not regular.

Now we are ready to state the regularity lemma.
\begin{thm}[Regularity lemma \rm{\cite[Theorem 17]{regular-lemmas2007}}] \label{Reg-lem}
Let $k\geq 2$ be a fixed integer. For all positive constants $\eta$ and $\delta_k$ and all functions $r:\mathbb{N} \rightarrow \mathbb{N}$ and $\delta:\mathbb{N} \rightarrow (0,1]$, there are integers $t$ and $n_0$ such that the following holds.
For every $k$-graph $H$ of order $n\ge n_0$ and $t!$ dividing $n$, there exists a family of partitions $\mathcal{P}= \mathcal{P}(k-1)$ of $V(H)$ such that\vspace{2mm}

$(1)$ $\mathcal{P}$ is $(\eta, \delta(t), t)$-equitable and

$(2)$ $H$ is $(\delta_k, r(t))$-regular w.r.t. $\mathcal{P}$.
\end{thm}

Note that the constants in Theorem~\ref{Reg-lem} can be chosen to satisfy the following hierarchy:
\[
\frac{1}{n_0}\ll \frac{1}{r}=\frac{1}{r(t)}, \delta=\delta(t)\ll \min\{\delta_k,1/t\}\ll \eta.
\]
Given $d\in (0,1)$, we say that an edge $e$ of $H$ is \emph{$d$-useful} if it lies in $\mathcal{K}_k (\hat{P}^{(k-1)})$ for some $\hat{P}^{(k-1)} \in \hat{\mathcal{P}}^{(k-1)}$ such that $H$ is $(d_k, \delta_k, r)$-regular w.r.t $\hat{P}^{(k-1)}$ for some $d_k\geq d$. If we choose $d\gg \eta$, then the following lemma will be helpful in later proofs.

\begin{lemma}\rm{(\cite[Lemma 4.4]{K2010Hamilton}).}\label{useful-edge}
At most $2dn^k$ edges of $H$ are not $d$-useful.
\end{lemma}

\subsection{Statement of a counting lemma.}
In our proofs we shall also use a counting lemma. Before stating this lemma, we need more definitions.

Suppose that $\mathcal{H}$ is an $(s, k)$-complex with vertex classes $V_1,\dots,V_s$, which all have size $m$. Suppose also that $\mathcal{G}$ is an $(s, k)$-complex with vertex classes $X_1,\dots, X_s$ of size at most $m$. We write $E_i(\mathcal{G})$ for the set of all $i$-edges of $\mathcal{G}$ and $e_i(\mathcal{G}):= |E_i(\mathcal{G})|$. We say that $\mathcal{H}$ \emph{respects the partition} of $\mathcal{G}$ if whenever $\mathcal{G}$ contains an $i$-edge with vertices in $X_{j_1}, \dots , X_{j_i}$, then there is an $i$-edge of $\mathcal{H}$ with vertices in $V_{j_1}, \dots , V_{j_i}$. On the other hand, we say that a labelled copy of $\mathcal{G}$ in $\mathcal{H}$ is \emph{partition-respecting} if for each $i \in[s]$ the vertices corresponding to those in $X_i$ lie within $V_i$. We denote by $|\mathcal{G}|_{\mathcal{H}}$ the number of labelled, partition-respecting copies of $\mathcal{G}$ in $\mathcal{H}$.

\begin{lemma}\label{Counting-lem}{\rm(Counting lemma {\cite[Theorem 4]{CFKO}}).} Let $k,s,r,t,n_0$ be positive integers and
let $d_2,\dots,d_k$, $\delta,\delta_k$ be positive constants such that $1/{d_i}\in \mathbb{N}$ and
\[
1/{n_0}\ll 1/r, \delta\ll \min\{\delta_k, d_2, \dots, d_{k-1} \}\ll \delta_k \ll d_k, 1/s, 1/t.
\]
Then the following holds for all integers $n\ge n_0$. Suppose that $\mathcal{G}$ is an $(s, k)$-complex on $t$ vertices with vertex classes $X_1,\dots, X_s$. Suppose also that $\mathcal{H}$ is a $(\mathbf{d},\delta_k,\delta, r)$-regular $(s, k)$-complex with vertex classes $V_1,\dots,V_s$ all of size $n$, which respects the partition of $\mathcal{G}$. Then
\[
|\mathcal{G}|_{\mathcal{H}}\ge \frac{1}{2}n^t \prod \limits_{i=2}^k d^{e_i(\mathcal{G})}_i.
\]
\end{lemma}
%[Restriction lemma~~\cite{K2010Hamilton}, Lemma 4.1]
\subsection{Statement of a restriction lemma.}
To prove Theorem~\ref{k-partite}, we need the following restriction lemma.
\begin{lemma}\label{Restriction-lem}{\rm(Restriction lemma {\cite[Lemma 4.1]{K2010Hamilton}}).}
Let $k, s, r, m$ be positive integers and $\gamma ,d_2, \dots ,$ $d_{k}, \delta , \delta _{k}$ be positive constants such that
\[
1/{m}\ll 1/r, \delta\leq \min\{\delta_k, d_2, \dots, d_{k-1} \}\leq \delta_k \ll \gamma \ll d_{k},1/s.
\]
Let $\mathcal{H}$ be a $(\mathbf{d},\delta_k,\delta, r)$-regular $(s,k)$-complex with vertex classes $V_1, \dots ,V_{s}$ of size $m$. For each $i$ let $V'_{i}\subseteq V_{i}$ be a set of size at least $\gamma m$. Then the restriction $\mathcal{H}' = \mathcal{H}[V'_1\cup \dots \cup V'_{s}]$ of $\mathcal{H}$ to $V'_1\cup \dots \cup V'_{s}$ is $(\mathbf{d},\sqrt{\delta _{k}},\sqrt{\delta },r)$-regular.
\end{lemma}

\subsection{The proof of Lemma~\ref{supersaturation}.}\label{section} In this subsection, we prove Lemma~\ref{supersaturation} using the hypergraph regularity method.
In order to prove Lemma~\ref{supersaturation}, we shall construct an auxiliary $k$-graph as follows. Given a $k$-graph $H$ and $w\in V(H)$, let $H_w$ be the $k$-graph obtained from $H$ by duplicating $w$ $n$ times and denote the set of the $n$ copies of $w$ by $V_w$ (namely, 
we add to $H$ $n$ new vertices so that each of them has exactly the same link $(k-1)$-graph as $w$). The auxiliary hypergraph $H_w$ will play a key role in showing that $F\in
\cover_k$.
We shall apply Theorem \ref{Reg-lem} to $H_w$ with the initial partition $V(H)\cup V_w$. Then we apply Lemma~\ref{Counting-lem} to a subhypergraph with good properties in $H_w$, which helps us to prove Lemma~\ref{supersaturation}. 
%We restate the lemma for convenience.

%\medskip
%\noindent\textbf{Lemma 2.3.}
%\emph{Given $F\in \cover_k$ with $f:=v(F)$, and $0<p,\alpha<1$, there exist $\mu,\eta>0$ and an integer $n_0\in \mathbb{N}$ such that the following holds.
%If $H$ is an $(n,p,\mu,\alpha)$ $k$-graph with $n\geq n_0$, then for any vertex $w$ in $V(H)$, $w$ is contained in at least $\eta n^{f-1}$ copies of $F$. }

\begin{proof}[Proof of Lemma~\ref{supersaturation}]
Given $p,\alpha>0$, let $\alpha'=\min\{\frac{p^2}{2(k-1)!}, \frac{\alpha}{8}\}$. Let $H$ be an $(n,p,\mu,\alpha)$ $k$-graph and set $V:=V(H)$.
Our goal is to apply the definition of $\cover_k$ in certain sub-hypergraph of $H$, i.e., taking $\mu>0$ small enough and $n$ large enough so that every vertex of a $((1-\alpha')n,p,\sqrt{\mu},\alpha')$ $k$-graph $H'$ is contained in a copy of $F$.
To prove this, we use the hypergraph regularity lemma. Now we introduce new constants satisfying the following hierarchy:
\[
\frac{1}{n_0}\ll \frac{1}{r}=\frac{1}{r(t)}, \delta=\delta(t)\ll \min\{\delta_k, d_2, \dots, d_{k-1},1/t\}\ll\delta_k, \eta\ll d\ll \mu\ll \alpha', 2^{-k}, 1/f
\]
where $d_i \geq 1/t$ and $1/d_i \in \mathbb{N}$ for all $i = 2, \dots , k-1$.
Recall that the hypergraph regularity lemma is proved by iterated refinements starting with an arbitrary initial partition. Hence, for a given vertex $w \in V$, we apply Theorem \ref{Reg-lem} to $H_w$ with the initial partition $V\cup V_w$. Then we obtain a family of partitions $\mathcal{P}=\{\mathcal {P}^{(1)},\dots, \mathcal {P}^{(k-1)}\}$ of $V\cup V_w$ and $H_w$ is $(\delta_k, r)$-regular w.r.t.~$\mathcal{P}$, where $|\mathcal {P}^{(1)}|=2a_1$ and $1/\eta \leq a_1 \leq t$. We may assume that $t!$ divides $n$ by discarding up to $t!$ vertices if necessary.
%(Here if necessary, we can delete some vertices to satisfy divisibility.)
%By the definition of $(\eta, \delta, t)$-equitable partition~\ref{eq-partition}, for every $k$-tuple $K \in \mathrm{Cross}_k$, the $(k, k-1)$-complex $\bigcup^{k-1}_{i=1}\hat{P}^{(i)}(K)$ is $(\mathbf{d}, \delta,\delta,1)$-regular.

We next delete from $H_w$ all edges which are not $d$-useful. By Lemma~\ref{useful-edge}, this results the removal of at most $2d(2n)^k$ edges. By averaging, there exists a cluster $W\subseteq V_w$ in $\mathcal {P}^{(1)}$ such that at most $2d(2n)^k/{a_1}\le 2\mu n^k$ edges are removed from the induced sub-hypergraph $H_w[V\cup W]$. Let $H'_w$ be the resulting $k$-graph after deleting those useless edges from $H_w[V\cup W]$. Since $\mu\ll \alpha'<\alpha$ and $\sum_{w'\in W}\deg_{H_w}(w')\ge (\alpha/{a_1}) n^{k}$, there is a vertex $w^*\in W$ such that $\deg_{H'_w}(w^*)\ge (\alpha/2)n^{k-1}$.
%In the following proof, due to $w^*$ is a clone of $w$, let $w^*$ replace $w$ in $H'_w[V]$ and
Set $H':= H'_w[V\setminus\{w\}\cup \{w^*\}]$. It follows that $\deg_{H'}(w^*)=\deg_{H'_w}(w^*)\ge (\alpha/2)n^{k-1}$ since $N_{H_w'}(w^*)\subseteq V\setminus \{w\}$.
Note that $H'$ is $(p,3\mu)$-dense, because $H$ is $(p,\mu)$-dense and $|E(H)|-|E(H')|\le 2\mu n^k$. Similarly, let $X$ be the set of vertices in $H'$ of degree less than $2\alpha' n^{k-1}$.
By the quasi-randomness, we have
\[
(k-1)! |X| 2\alpha'n^{k-1}\ge e_{H'}(X, V,\dots, V) \ge p|X|n^{k-1} - 3\mu n^k,
\]
which implies that $|X|\le \sqrt{\mu} n\le \alpha' n$ by the choice of $\alpha'$.
%Note that due to the special role of $w$ (i.e. every verex in cluster $V_k$ is the copy of $w$), we have $\deg_{H'}(w)\ge (\alpha/4)n^{k-1}$.
Let $H''$ be the induced sub-hypergraph of $H'$ of size exactly $(1-\alpha')n$ on a subset of $V\setminus X$. Then $w^*$ is still in $H''$ as $w^*\notin X$.
Since deleting $\alpha' n$ vertices will reduce the degree of each  remaining vertex by at most $\alpha' n^{k-1}$, we conclude that $H''$ is a $((1-\alpha')n,p,\sqrt{\mu},\alpha')$ $k$-graph by the choice of $\alpha'$.
By the definition of $\cover_k$, $H''$ contains a copy $F^*$ of $F$ covering $w^*$. Furthermore, due to our construction of $H''$, each edge of $F^*$ is $d$-useful. Without loss of generality, suppose that the vertices of $F^*$ lie in $s$ clusters, where $1<s\le f$.

Let $\mathcal{F}^{*\le}$ be an $(s,k)$-complex turned from $F^*$ by making every edge into a complete $i$-graph $K^{(i)}_k$ for each $1\le i\le k$.
Next, we define $\mathcal{H}^*$ to be the $(s,k)$-complex obtained from the $(s,k-1)$-complex $\bigcup_{e\in E(F^*)} \bigcup_{i=1}^{k-1}
\hat{P}^{(i)}(e)$ by adding $E(H'_w)\cap \bigcup_{e\in E(F^*)}\mathcal{K}(\hat{P}^{(k-1)}(e))$ as the ``$k$th level".
Since every edge of $H'_w$ is $d$-useful, we know that for each $e\in E(F^*)$, $H'_w$ is $(d_k,\delta_k,r )$-regular w.r.t. $\hat{P}^{(k-1)}(e)$ for some $d_k\ge d$. Note that these polyads ``fit together''. By this we mean that if edges $e$ and $e'$ of $F^*$ intersect in $j$ vertices, then
\[(\bigcup\limits_{i=1}^{k-1} \hat{P}^{(i)}(e) )\cap(\bigcup\limits_{i=1}^{k-1} \hat{P}^{(i)}(e') ) = \bigcup\limits_{i=1}^{j} \hat{P}^{(i)}(e\cap e').\]
Then $\mathcal{H}^*$ is a $(\mathbf{d}, d_k, \delta_k, \delta, r)$-regular $(s,k)$-complex, where $\mathbf{d}=(d_2, \dots , d_{k-1})$ is as in the definition of an $(\eta, \delta, t)$-equitable partition. (Here we may assume a common density $d_k$ for the $k$th level by applying the slicing lemma (\cite{regular-lemmas2007}, Proposition 22) if necessary.) Furthermore, $\mathcal{H}^*$ respects the partition of the complex $\mathcal{F}^{*\le}$.
%$(\mathbf{d}, d_k, \delta_k, \delta, r)$-regular
 Thus we apply Lemma~\ref{Counting-lem} on $\mathcal{F}$ and $\mathcal{H}^*$ and obtain
\[
|\mathcal{F}^{*\le}|_{\mathcal{H}^*}\ge \frac{1}{2}\left(\frac{n}{a_1}\right)^f \prod \limits_{i=2}^k d^{e_i(\mathcal{F}^{*\le})}_i=\frac{\eta'}{2a_1}n^f, \ \text{where} \ \eta'=a_1^{1-f}\prod \limits_{i=2}^k d^{e_i(\mathcal{F}^{*\le})}_i.
\]
Furthermore, as at most $n^{f-1}$ copies of $F$ may contain $w$, there are at least $\frac{\eta'}{3a_1}n^f$ copies of $F$ in $H'_w$ which do not contain $w$.
Since every vertex in cluster $W$ is a clone of $w$ and $|W|=n/{a_1}$, there exists $w'\in W$ contained in at least $\eta n^{f-1}$ copies of $F$ that do not contain $w$, where $\eta:= \eta'/3$. As in each of these copies of $F$, $w'$ can be replaced by $w$, we are done.
\end{proof}

\subsection{The proof of the backward implication of Theorem~\ref{k-partite}} In this subsection we prove the backward implication of Theorem~\ref{k-partite} using Theorem~\ref{k-graph}, the regularity lemma, the counting lemma and the restriction lemma. The forward implication of Theorem~\ref{k-partite} will be proved in Section 5. We just state the statement of the backward implication of Theorem~\ref{k-partite} for convenience.

\medskip
\noindent\textbf{The backward implication of Theorem 1.4.}
\emph{For $k \geq 3$ and a $k$-partite $k$-graph $F$, if there exists $v^*\in V(F)$ such that $|e\cap e'|\leq 1$ for any two edges $e,e'$ with $v^*\in e$ and $v^*\notin e'$, then $F\in \factor_k$.}
\medskip

We first give an outline of our proof.
Given a vertex $w$ in the $(n,p,\mu ,\alpha )$ $k$-graph $H$, we need to find a copy of $F$ that maps $v^*$ to $w$.
We also use the $k$-graph $H_w$ as in the previous proof, where the vertex $w$ is duplicated $n$ times.
After applying the regularity lemma to $H_w$, our goal is to define an appropriate regular (and dense) complex which would allow us to find the desired embedding of $F$ by applying the counting lemma.
To achieve this, we first pick a regular (and dense) $(k,k)$-complex $\mathcal{H}^*_1$ with clusters $V_0, V_1, \dots, V_{k-1}$ where $V_0$ is a cluster all of whose vertices are clones of $w$.
Next by the $(p,\mu)$-denseness, pick an arbitrary cluster $V_k$ and we can pick a regular (and dense) $(k,k)$-complex $\mathcal{H}^*_2$ with clusters $V_1, \dots, V_{k}$.
We further cut these clusters into smaller pieces $V_i^j$, $0\le i\le k$, $1\le j\le f$ and map each vertex of $F$ to a small piece, according to the $k$-partition of $F$ while mapping $v^*$ to $V_0^1$.
The fact that both $\mathcal{H}^*_1$  and $\mathcal{H}^*_2$ are regular and dense passes to the combinations of the smaller pieces by the restriction lemma.
The key property of $F$ guarantees that we can choose a regular and dense $(v(F), k)$-complex so that the parts coming from $\{V_0,\dots, V_{k-1}\}$ and $\{V_1, \dots, V_{k}\}$ can be safely combined together without conflict (e.g., there would not be clusters $V_i^j$, $V_{i'}^{j'}$ and edges $e$ containing $v^*$, $e'$ not containing $v^*$ such that $e$ and $e'$ require different cells from $(V_i^j, V_{i'}^{j'})$).

\begin{proof}
We prove the backward implication of Theorem~\ref{k-partite} here. Suppose that $F$ is a $k$-partite $k$-graph with vertex classes $X_1,\dots ,X_k$ and satisfies the property stated in Theorem~\ref{k-partite}. Let $|X_i|=f_i$ for $i\in [k]$ and $f = \max_{1\le i\le k} f_i$. For convenience, we label $V(F)$ and denote the $j$th vertex in $X_i$ by $v_{i}^{j}$ for $1\leq i \leq k,1\leq j \leq f_i$. Without loss of generality, we may assume that $v^*$ is labeled by $v_1^1$. By Theorem~\ref{k-graph}, we only need to show that $F\in\cover_k$. In other words, our goal is to show that given $p,\alpha >0$, we take $\mu >0$ small enough and $n$ large enough such that every vertex $w$ in an $(n,p,\mu ,\alpha )$ $k$-graph $H$ is contained in a copy of $F$. Choose constants satisfying the following hierarchy:
\[
1/n_0 \ll 1/r, \delta \ll \min \{\delta_k, d_2, \dots, d_{k-1} \} \ll \delta _{k} \ll \gamma \ll d \ll \mu \ll p, \alpha ,1/f
\] where $\mathbf{d}= (d_2, \dots , d_{k-1})$ such that $d_i \geq 1/t$ and $1/d_i \in \mathbb{N}$ for all $i = 2, \dots , k-1$. As in the proof of Lemma~\ref{supersaturation}, for a given vertex $w\in V(H)$, we apply Theorem \ref{Reg-lem} to $H_w$ and obtain a family of partitions $\mathcal{P}=\{\mathcal {P}^{(1)},\dots, \mathcal {P}^{(k-1)}\}$ of $V\cup V_w$ such that $H_w$ is $(\delta_k, r)$-regular w.r.t. $\mathcal{P}$, where $|\mathcal{P}^{(1)}|=2a_1$ and $1/ \eta \leq a_1 \leq t$. We may assume that $t!$ divides $n$.
%By the definition of $(\elta, )$-equitable partition~\ref{eq-partition}, there exists $\mathbf{d}= (d_2, \dots , d_{k-1})$ such that $d_i \geq 1/t$ and $1/d_i \in \mathbb{N}$ for all $i = 2, \dots , k-1$. For every $k$-tuple $K \in \mathrm{Cross}_k$, the $(k, k-1)$-complex $\bigcup^{k-1}_{i=1}\hat{P}^{(i)}(K)$ is $(\mathbf{d}, \delta,\delta,1)$-regular.

Furthermore, by the proof of Lemma~\ref{supersaturation}, we can find a cluster $V_0 \subseteq V_{w}$ such that  there is a $d$-useful edge $e_1$ in $H_w[V\cup V_0]$ with $|e_1\cap V_0| = 1$.
%By Lemma~\ref{useful-edge} and $d\ll \mu $, there are at most $2d(2n)^{k}\leq 2\mu n^k$ not $d$-useful edges in $H_w$. So there exists a cluster $V_0 \subseteq V_{w}$ such that there are at most $\frac{2\mu n^k}{a_1}$ not $d$-useful edges in $H_w[V\cup V_0]$.
%On the other hand, we have $\sum \limits_{w'\in V_0} \deg_{H_w}(w') = \frac{n}{a_1}\deg_{H}(w)\geq \frac{n}{a_1}\alpha n^{k-1}$ as $V_0$ is a set of $\frac{n}{a_1}$ copies of $w$. Since $\mu \ll \alpha $, we have $\frac{\alpha n^{k}}{a_1} > \frac{2\mu n^k}{a_1}$.
Recalling the definition of $d$-useful edges, this means that there is a polyad $\hat{P}^{(k-1)}(e_1) \in \hat{\mathcal{P}}^{(k-1)}$ such that $H_w$ is $(d'_k, \delta_k, r)$-regular w.r.t $\hat{P}^{(k-1)}(e_1)$ for some $d'_k\geq d$. We assume that $\hat{P}^{(k-1)}(e_1)$ lies in $k$ clusters written as $V_0, V_1, \dots ,V_{k-1}$. Choose an arbitrary cluster from $V\setminus \bigcup_{i=1}^{k-1}V_i$ written as $V_{k}$. As $H$ is an $(n,p,\mu ,\alpha )$ $k$-graph, we have $e(V_1, V_2, \dots ,V_{k})\geq p(\frac{n}{a_1})^{k} - \mu n^{k}$. Therefore, there are at least $p(\frac{n}{a_1})^{k} - 3\mu n^{k}>0$ useful edges in $E(V_1, V_2, \dots ,V_{k})$,
%We count edges in $E(V_1, V_2, \dots ,V_{k})$ that may lie in a polyad $\hat{P}^{(k-1)} \in \hat{\mathcal{P}}^{(k-1)}$ such that $|E(H)\cap \mathcal{K}_{k}(\mathcal{P}^{(k-1)})|\leq d|\mathcal{K}_{k}(\mathcal{P}^{(k-1)})|$ or $H$ is not $(\delta _{k},r)$-regular with respect to $\hat{P}^{(k-1)}$. There are at most $d(\frac{n}{a_1})^{k}$ edges of the first type and at most $\delta _kn^{k}$ edges of the second type.
which implies that there is a polyad $\hat{P}^{(k-1)}(e_2) \in \hat{\mathcal{P}}^{(k-1)}$ such that $H_w$ is $(d''_k, \delta_k, r)$-regular w.r.t $\hat{P}^{(k-1)}(e_2)$ for some $d''_k\geq d$.

Now we begin to embed $F$. Our aim is to embed $v^*$ into $V_0$, $X_1\setminus \{v^*\}$ into $V_k$ and $X_i$ into $V_{i-1}$ for $2\leq i\leq k$ via Lemma~\ref{Counting-lem}. For $\ell=1,2$, we define $\mathcal{H}^{*}_\ell$ to be the $(k,k)$-complex obtained from the $(k,k-1)$-complex $\bigcup _{i=1}^{k-1}\hat P^{(i)}(e_\ell)$ by adding $E(H_w)\cap \mathcal{K}(\hat P^{(k-1)}(e_\ell))$ as the  ``$k$th level". 
Then $\mathcal{H}^{*}_\ell$ is a $(\mathbf{d},d_k,\delta _k,\delta ,r)$-regular $(k,k)$-complex, where $\mathbf{d}=(d_2, \dots , d_{k-1})$ is as in the definition of an $(\eta, \delta, t)$-equitable partition. (Here we may assume a common density $d_k$ for the $k$th level by applying the Slicing  Lemma (\cite{regular-lemmas2007}, Proposition 22) if necessary.) 
Now we cut each $V_0, V_1, \dots ,V_{k}$ into $f$ pieces equitably and denote each piece by $V^j_i$ where $1 \leq j \leq f$ for each $0\leq i \leq k$. By Lemma~\ref{Restriction-lem}, the restriction $\mathcal{H}^{*}_1[V^{j_0}_0, V^{j_1}_1, \dots ,V^{j_{k-1}}_{k-1}]$ of $\mathcal{H}^{*}_1$ is $(\mathbf{d}, d_k,\sqrt{\delta _k},\sqrt{\delta} ,r)$-regular for all $1\leq j_0,\dots ,j_{k-1}\leq f$. (Since $\gamma \ll 1/f$, $|V^{j_i}_i|\geq |V_i|/f\geq \gamma |V_i|$ holds for all $0\leq i\leq k-1$.)
Similarly we can apply Lemma~\ref{Restriction-lem} to $\mathcal{H}^{*}_2$ as well. 
We plan to embed $v^*$ into $V^1_0$, $v_1^{j}$ into $V^{j}_k$ for $2\leq j\leq f_1$ and $v_{i}^{j}$ into $V^{j}_{i-1}$ for $2\leq i \leq k,1\leq j \leq f_i$. Let $\mathcal{F}^{\leq}$ be a $k$-complex obtained from $F$ by making every edge a complete $i$-graph $K^{(i)}_k$ for each $1\leq i\leq k$. If $e = \{v^*,v_2^{j_2}, \dots ,v_k^{j_k}\}\in E(F)$ for $2\le i \le k$ and $1\leq j_{i} \leq f_{i}$,
then we choose the corresponding vertex subset $V_0^1,V_{i-1}^{j_{i}}\subseteq V(\mathcal{H}^{*}_1)$. 
By Lemma~\ref{Restriction-lem} and $d_k>0$, there is an edge $\hat e\in H_w[V_0^1,V_1^{j_2},\dots ,V_{k-1}^{j_k}]$. 
Let $\mathcal{C}(e):=\bigcup _{i=1}^{k-1}\hat P^{(i)}_{res}(\hat e)$ be the restriction of $\bigcup _{i=1}^{k-1}\hat P^{(i)}(\hat e)$ to $V_0^1$ and $V_{\ell-1}^{j_{\ell}}$. If $e = \{v_1^{j_1},v_2^{j_2}, \dots ,v_k^{j_k}\}\in E(F)$ for $2\leq j_1 \leq f_1$ and $1\leq j_{i} \leq f_{i}$,
 then we choose the corresponding vertex subsets $V_k^{j_1},V_1^{j_2},\dots ,V_{k-1}^{j_k}\subseteq V(\mathcal{H}^{*}_2)$.
 Similarly, there is an edge $\hat e\in H_w[V_k^{j_1},V_1^{j_2},\dots ,V_{k-1}^{j_k}]$ and we define $\mathcal{C}(e):=\bigcup _{i=1}^{k-1}\hat P^{(i)}_{res}(\hat e)$ as the restriction of $\bigcup _{i=1}^{k-1}\hat P^{(i)}(\hat e)$.
The key point here is that for any two edges $e,e'$ with $v^*\in e$ and $v^*\notin e'$, $|e\cap e'|\leq 1$, from which we infer that $\mathcal C(e)$ and $\mathcal C(e')$ do not each contain a cell in a pair of clusters, say, $V_{i_1}^{j_1}$ and $V_{i_2}^{j_2}$ (see Figure~1). Then let $\mathcal{H}^{*}$ be the $(v(F),k)$-complex obtained from the $(v(F),k-1)$-complex $\bigcup \limits _{e\in F}\mathcal{C}(e)$ by adding the edges it supports in $H_w$
 %$\bigcup \limits _{e\in F}(E(H)\cap \mathcal{K}(\hat P^{(k-1)}(e)))$
 as the ``$k$th level", and it is easy to verify that $\mathcal{H}^{*}$ is a $(\mathbf{d},d_k,\sqrt{\delta _k},\sqrt{\delta },r)$-regular $(v(F),k)$-complex.
 Due to the construction of $\mathcal{H}^{*}$ and Lemma~\ref{Restriction-lem}, $\mathcal{H}^{*}$ respects the partition of the complex $\mathcal{F}^{\leq}$. By Lemma~\ref{Counting-lem}, we obtain a copy of $F$ containing $w$ and we are done.
 %can embed $\mathcal{F}^{\leq}$ into $\mathcal{H}^{*}$ where $v$ is embedded into $V^1_0$, $v_1^{\ell}$ is embedded into $V^{\ell}_k$ for $2\leq \ell \leq f_1$ and $v_{i}^{j}$ is embedded into $V^{j}_{i-1}$ for $2\leq i \leq k,1\leq j \leq f_i$.
\end{proof}

%The forward implication of Theorem~\ref{k-partite} will be proved in Section 5.

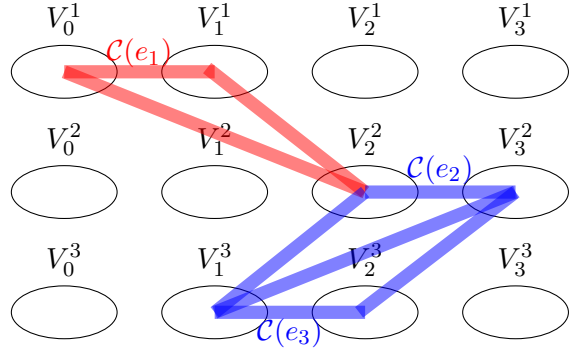
\begin{figure}\label{figure-1}
\begin{center}
\begin{tikzpicture}
[inner sep=2pt,
   vertex/.style={circle, draw=blue!50, fill=blue!50},
   ]
\draw (0,3.2) ellipse (20pt and 10pt);
\draw (0,1.6) ellipse (20pt and 10pt);
\draw (0,0) ellipse (20pt and 10pt);
\draw (2,3.2) ellipse (20pt and 10pt);
\draw (2,1.6) ellipse (20pt and 10pt);
\draw (2,0) ellipse (20pt and 10pt);
\draw (4,3.2) ellipse (20pt and 10pt);
\draw (4,1.6) ellipse (20pt and 10pt);
\draw (4,0) ellipse (20pt and 10pt);
\draw (6,3.2) ellipse (20pt and 10pt);
\draw (6,1.6) ellipse (20pt and 10pt);
\draw (6,0) ellipse (20pt and 10pt);
\node at (0,3.9) {$V_0^1$};
\node at (0,2.3) {$V_0^2$};
\node at (0,0.7) {$V_0^3$};
\node at (2,3.9) {$V_1^1$};
\node at (2,2.3) {$V_1^2$};
\node at (2,0.7) {$V_1^3$};
\node at (4,3.9) {$V_2^1$};
\node at (4,2.3) {$V_2^2$};
\node at (4,0.7) {$V_2^3$};
\node at (6,3.9) {$V_3^1$};
\node at (6,2.3) {$V_3^2$};
\node at (6,0.7) {$V_3^3$};
\draw[line width=5pt,red,opacity=0.5](0,3.2) -- (2,3.2);
\draw[line width=5pt,red,opacity=0.5](1.9,3.25) -- (4,1.6);
\draw[line width=5pt,red,opacity=0.5](0,3.2) -- (4,1.6);
\draw[line width=5pt,blue,opacity=0.5](2,0) -- (4,0);
\draw[line width=5pt,blue,opacity=0.5](3.92,0) -- (6,1.6);
%\draw[line width=5pt,blue,opacity=0.5](2,1.6) -- (6,0);
\draw[line width=5pt,blue,opacity=0.5](4,1.6) -- (6,1.6);
\draw[line width=5pt,blue,opacity=0.5](2,0) -- (4,1.6);
\draw[line width=5pt,blue,opacity=0.5](2,0) -- (6,1.6);

\node[red] at (1,3.45) {$\mathcal{C}(e_1)$};
\node[blue] at (5,1.9) {$\mathcal{C}(e_2)$};
\node[blue] at (3,-0.27) {$\mathcal{C}(e_3)$};
%
%\node at (0,2) [vertex, label=left:$x_1$] {};
%\node at (0,1) [vertex, label=left:$y_1$] {};
%\node at (0,0) [vertex, label=left:$z_1$] {};
%\node at (3,2) [vertex, label=right:$x_2$] {};
%\node at (3,1) [vertex, label=right:$y_2$] {};
%\node at (3,0) [vertex, label=right:$z_2$] {};
%\draw (0,2) -- (3,0);
%\draw (0,1) -- (3,1);
%\draw (0,0) -- (3,2);
%\draw[dashed] (0,0) -- (3,0);
%\draw[dashed] (0,1) -- (3,0);
%\draw[dashed] (0,0) -- (3,1);
\end{tikzpicture}

\caption{An illustration of the proof of Theorem~\ref{k-partite} for the case $k=3$ and $f=3$.}
\end{center}
\end{figure}

%%%%%%%%%%%%%%%%%%%%%%%%%%%%%%%%%%%%%%%%%%%%%%%%%%%%%%%%%%%%%%%%%%%%%%%%%%%%%%%
%%%%%%%%%%%%%%%%%%%%%%%%%%%%%%%%%%%%%%%%%%%%%%%%%%%%%%%%%%%%%%%%%%%%%%%%%%%%%%%
\section{Proof of the backward implication of Theorem \ref{equal}}
Let $F$ be a $3$-graph with $v(F):=f$ satisfying conditions {\rm (i)} and {\rm (ii)} in Theorem \ref{equal}. Further, we assume that  $\mathcal{P}=\{X,Y,\{v^*\}\}$ is a certain partition of $V(F)$ satisfying condition {\rm (ii)}.
By Theorem~\ref{k-graph}, it suffices to show that $F\in \cover_3$.
In this section, we prove $F\in \cover_3$ using the regularity lemma for $3$-graphs.

 Let $\{v_1,v_2,\dots,v_f\}$ be an enumeration of $V(F)$ satisfying condition {\rm (b)} in Theorem \ref{0}. Then it admits the property stated as follows.
\begin{observation}\label{3-coloring}
For any $v\in V(F)$ and $1\leq i<j<k\leq f$,  we have $\{v,v_i,v_j\} \notin E(F)$ or $\{v,v_j,v_k\} \notin E(F)$.
\end{observation}
\begin{proof}
Suppose that both $\{v,v_i,v_j\} \in E(F)$ and $\{v,v_j,v_k\} \in E(F)$ and $v=v_r$ ($r\neq j$). If $r<j$, then $\{v_r,v_j,v_k\}\in E(F)$ implies that $\varphi(v_r,v_j)={\color{red} red}$ while $\{v_r,v_i,v_j\}\in E(F)$ implies that $\varphi(v_r,v_j)={\color{blue} blue}~or~{\color{green}green}$. If $r>j$, then similarly $\{v_r,v_j,v_k\}\in E(F)$ implies that $\varphi(v_r,v_j)={\color{red} red}~or~{\color{blue}blue}$ while $\{v_r,v_i,v_j\}\in E(F)$ implies that $\varphi(v_r,v_j)={\color{green} green}$. Therefore, either $\{v,v_i,v_j\} \notin E(F)$ or $\{v,v_j,v_k\} \notin E(F)$.
\end{proof}

Let $F_{v^*}\subseteq F$ be the $3$-graph consisting of all $3$-edges containing $v^*$. For any integer $0\leq i\leq3$, let $F_i\subseteq F\setminus\{v^*\}$ be the $3$-graph consisting of all $3$-edges $e$ with $|e\cap X|=i$. Then an easy observation is that $\partial F_{v^*}$, $\partial F_0$, $\partial F_1$, $\partial F_2$ and $\partial F_3$ are pairwise disjoint. In the following lemma, we show that $\partial F$ actually admits a more structural $3$-coloring w.r.t. the partition $\mathcal{P}$.

\begin{lemma}\label{alter} There is an enumeration $\{v_1,v_2,\cdots,v_f\}$ of $V(F)$ with $v_1=v^*$, $\{v_2,\cdots,v_{|X|+1}\}=X$ and $\{v_{|X|+2},\cdots,v_f\}=Y$ satisfying {\rm (b)} in Theorem \ref{0}.

\end{lemma}
\begin{proof}
Let $\tau$ be any enumeration of $V(F)$ satisfying \rm(b). We now construct a new enumeration $\{v_1,v_2,\cdots,v_f\}$ of $V(F)$ by letting $v_1=v^*$, $\{v_2,\cdots,v_{|X|+1}\}=X$ and $\{v_{|X|+2},\cdots,v_f\}=Y$, in which the enumerations of $X$ and $Y$ respect $\tau$ (namely, they are ordered as given by the sub-sequence of $\tau$). Since $\partial F_{v^*}$, $\partial F_0$, $\partial F_1$, $\partial F_2$ and $\partial F_3$ are pairwise disjoint,  it suffices to show that we can define the desired $3$-coloring $\varphi$ for $\partial F_{v^*}$ and for each $\partial F_i~(0\leq i\leq3)$. The combination of these is
a  $3$-coloring $\varphi$ of $\partial F$.

For $F_0$ and $F_3$, since $\tau$ satisfies \rm(b) and the enumerations of $X$ and $Y$ are consistent with $\tau$, we can always color $\partial F_0$ and $\partial F_3$ as desired. Next, we claim that $\partial F_1$ also admits a desired $3$-coloring. We first note that all edges in $\partial F_1$ lying in $Y$ can be always painted {\color{green} green}. Therefore, failing to color $\partial F_1$ implies that there is a vertex $v_i\in X$ and a $3$-edge $\{v_i,v_j,v_k\}\in F_1$ that forces us to paint $\{v_i,v_j\}$ {\color{red}red} while another $3$-edge $\{v_i,v_{\ell},v_j\}\in F_1$ forces us to paint $\{v_i,v_j\}$ {\color{blue}blue}. Then $\ell<j<k$ follows from the construction of $\varphi$. Since $v_j,v_{\ell},v_{k}$ are elements from $Y$ and the enumeration of $Y$ follows $\tau$, a contradiction with Observation~\ref{3-coloring} is obtained. The existence of the desired $3$-colorings for $\partial F_2$ and $\partial F_{v^*}$ is promised by a similar argument.
\end{proof}

\subsection{The regularity lemma for $3$-graphs.}
A key tool in our proof is the regularity lemma for $3$-graphs.
Here we use the version of regularity lemma for $3$-graphs stated in~\cite{Reiher_2018}.
In addition, a new result concerning ``cleaning" the regular partitions will be given. Now we first introduce the necessary notation in the following.

For two disjoint sets $X$ and $Y$, we denote by $K(X,Y)$ the \emph{complete bipartite} graph between $X$ and $Y$. Suppose that $V$ is a finite set of vertices and $\mathcal{P}(2)=\{\mathcal{P}^{(1)},\mathcal{P}^{(2)}\}$ is a family of partition of $V$ with $\mathcal{P}^{(1)}=\{V_1, \dots , V_{t}\}$. Recall the definition of ``cell" and ``polyad" in Section~\ref{section-equ-partition}. For every $\{x,y\}\in \mathrm{Cross}_2$ with $x\in V_i$ and $y\in V_j$ for $1\le i<j\le t$, we have
\[
\hat{P}^{(1)}=\hat{P}^{(1)}(\{x,y\})=P^{(1)}(x)\cup P^{(1)}(y)=V_i\cup V_j.
\]
Thus, $\mathcal{K}_2(\hat{P}^{(1)})=K(V_i,V_j)$ and every cell $P^{(2)}\subseteq \mathcal{K}_2(\hat{P}^{(1)})$ is a bipartite subgraph. To ease notation, we denote by $P^{ij}_{a}=(V_i\dot\cup V_j,E^{ij}_{a})$ a specific cell in $K(V_i,V_j)$ and let $P^{ijk}_{abc}=P^{ij}_{a}\cup P^{ik}_{b}\cup P^{jk}_{c}$ denote a specific polyad over $2$-graphs, i.e. $ P^{ijk}_{abc}=(V_i\dot\cup V_j\dot\cup V_k,E^{ij}_{a}\cup E^{ik}_{b}\cup E^{jk}_{c})$.

We state the regularity lemma for $3$-graphs in~\cite{Reiher_2018} as follows.
\begin{lemma}\label{RL}{\rm({\cite[Theorem 3.2]{Reiher_2018}}).}
For all $\delta_3>0$, $\delta_2:\mathbb{N}\longmapsto(0,1]$ and $t_0\in\mathbb{N}$, there exists an integer $T_0$ such that for every $n\geq t_0$ and every $3$-graph $H=(V,E)$ on $n$ vertices the following holds.

There are integers $t$ and $\ell$ with $t_0\leq t\leq T_0$ and $\ell\leq T_0$, and there exists a partition $\mathcal{P}^{(1)}=\{V',V_1, \dots , V_{t}\}$ on $V$, and for all $1\leq i<j\leq t$ there exists a partition
\[\mathcal{P}^{ij}=\{P^{ij}_{a} : 1\leq a\leq \ell\}\]
of the edge set of the complete bipartite graph $K(V_i,V_j)$
satisfying the following properties
\begin{itemize}
  \item[{\rm (i)}] $|V'|\leq\delta_3n$ and $|V_1|=|V_2|=\cdots=|V_t|$;
  \item[{\rm (ii)}] for all $1\leq i<j\leq t$ and $ a\in[\ell]$ the bipartite graph $P^{ij}_{a}$ is $(1/\ell,\delta_2(\ell))$-regular;
  \item[{\rm (iii)}] $H$ is $(\delta_3,1)$-regular w.r.t. all but at most $\delta_3t^3\ell^3$ polyads $P^{ijk}_{abc}$ with $1\leq i<j<k\leq t$ and $a,b,c\in[\ell]$.
\end{itemize}
\end{lemma}

\noindent \textbf{Remark.} For convenience, we say that  a polyad $P^{ijk}_{abc}$ is \textit{bad} if $H$ is not $(\delta_3,1)$-regular \textit{w.r.t.} $P^{ijk}_{abc}$, otherwise it is \textit{good}. Moreover, a triple $\{i,j,k\}\in [t]^3$ (or $\{V_i,V_j,V_k\}$) is \textit{bad} if there are at least $\sqrt{\delta_3}\ell^3$ bad polyads in $\{V_i,V_j,V_k\}$, otherwise it is \textit{good}. Since by Property ${\rm (iii)}$ globally $H$ is not $(\delta_3,1)$-regular for up to at most $\delta_3t^3\ell^3$ polyads, an easy averaging argument shows that there are at most $\sqrt{\delta_3}t^3$ bad triples in the vertex partition provided by Lemma~\ref{RL}.

In order to prove $F\in \cover_3$, we shall apply Lemma \ref{RL} to the auxiliary $3$-graph $H_w$ defined as before in Section~\ref{section}, by duplicating $w$ $n$ times for a $3$-graph $H$ and $w\in V(H)$. Before giving the result concerning ``cleaning" the regular partitions, we need the following lemma which is an refined version of the~\cite[Corollary 2.2]{greenhill2017average} in a slightly strengthened form. The proof of this lemma can be found in~\cite[Theorem 2.1]{greenhill2017average}.
%and~\cite[Section 3]{MR2744811}.

\begin{lemma}\label{probability-lemma}
Let $[N]^{r}$ be the set of $r$-subsets of $\{1,\dots, N\}$ and let $f:[N]^{r}\longmapsto \mathbb{R}$ be given. Let $R$ be an uniformly random element of $[N]^{r}$. Suppose that $r<N/2$ and there exists $\zeta> 0$ such that
\[
|f(R')-f(R'')|\leq \zeta
\]
for any $R',R''\in [N]^{r}$ with $|R'\cap R''|=r-1$. Then for any real $t>0$,
\[
\mathbb{P} \left(f(R)-\mathbb{E}(f(R))\geq t\right)\leq \exp\left(-\frac{2t^2}{r\zeta^2}\right),
\]
and
\[
\mathbb{P} \left(f(R)-\mathbb{E}(f(R))\leq -t\right)\leq \exp\left(-\frac{2t^2}{r\zeta^2}\right). \qed
\]
\end{lemma}

Similarly as in other proofs based on the regularity method it will be convenient to
“clean” the regular partition provided by Lemma~\ref{RL}.
\begin{lemma}\label{clean}
For $\alpha\in (0,1)$, $0<\delta_3 <d<\frac{\alpha}{8}$ and $\delta_2:\mathbb{N}\longmapsto(0,1]$ with $\delta_2(x)\ll \frac{1}{x}$, there exist integers $T_0$, $n_0$ and $m_0$ such that for any integer $m\ge m_0$ and every $3$-graph $H=(V,E)$ on $n$ vertices with $n\geq n_0$ and $\delta(H)\geq \alpha n^2$ the following holds.

For any $w\in V$, there exists a sub-hypergraph $\hat{H}=(\hat{V},\hat{E})\subseteq H_w$, an integer $\ell\leq T_0$, a partition $\mathcal{P}^{}(1)=\{V_0, V_1,\cdots, V_m\}$ on $\hat{V}$ with $V_0\subseteq V_w$ and $\cup_{i=1}^{m}V_i\subseteq V$, and a partition $\mathcal{P}^{ij}=\{P^{ij}_{a}:1\leq a\leq \ell\}$ of $K(V_i,V_j)$ for all integers $i,j$ with $0\leq i<j\leq m$ satisfying the following properties
\begin{itemize}
  \item[{\rm (i)}] $|V_0|=|V_1|=|V_2|=\cdots=|V_m|\geq2(1-\delta_3)n/T_0$;\label{clean-1}
  \item[{\rm (ii)}] for all $0\leq i<j\leq m$ and $a\in[\ell]$ the bipartite graph $P^{ij}_{a}$ is $(1/\ell, \delta_2(\ell))$-regular;
  \item[{\rm (iii)}] $\hat{H}$ is $\delta_3$-regular w.r.t.~all polyads $P^{ijk}_{abc}$ with $0\leq i<j<k\leq m$ and $a,b,c\in[\ell]$, and $d(\hat{H}|P^{ijk}_{abc})$ is either $0$ or at least $d$;
  \item[{\rm (iv)}] for every $0\leq i< j< k\leq m$ we have
      $$|E_{\hat{H}}(V_i,V_j,V_k)|\geq |E_{H_w}(V_i,V_j,V_k)|-2d|V_i||V_j||V_k|;$$
  \item[{\rm (v)}]there are at least $\alpha m^2/4$ pairs $\{i,j\}$ with $1\leq i<j\leq m$ such that $d(\hat{H}|P^{0ij}_{abc})\geq d$ holds for at least one good polyad $P^{0ij}_{abc}$.\label{clean-5}
\end{itemize}
\end{lemma}

\begin{proof} Given $\alpha>0$ and $0<\delta_3 <d<\frac{\alpha}{8}$, let $H=(V,E)$ be an $n$-vertex $3$-graph with $\delta(H)\geq \alpha n^2$. For the proof of Lemma~\ref{clean}, we shall apply the
regularity lemma (Lemma \ref{RL}) with $\delta'_3$ sufficiently small such that
\[64\delta'_3<\delta_3^2 \quad \text{and} \quad48\sqrt{\delta'_3}\binom{m}{3}<\frac{1}{3}
\]
and the integer $t_0=\max\{m,\lceil\frac{1}{\delta'_3}\rceil\}$ and the given function $\delta_2:\mathbb{N}\longmapsto(0,1]$.

Given $w\in V$, we apply Lemma \ref{RL} to $H_w$ with the initial partition $V\cup V_w$. Now we obtain a partition $\mathcal{P}^{(1)}=\{V',V_1, \dots , V_{2t}\}$ of $V\cup V_w$ such that $|V'|\leq2\delta'_3n$ and $|V_1|=\cdots=|V_{2t}|$ with $t\ge t_0$. Clearly, for all $1\leq i<j\leq 2t$ and $ a\in[\ell]$, the bipartite graph $P^{ij}_{a}$ is $(1/\ell,\delta_2)$-regular. Moreover, $H_w$ is $(\delta'_3,1)$-regular w.r.t. all but at most $\delta'_3(2t)^3\ell^3$ polyads $P^{ijk}_{abc}$ with $1\leq i<j<k\leq 2t$ and $a,b,c\in[\ell]$. Without loss of generality, let $\bigcup_{i=1}^{t}V_i\subseteq V$. By the remark of Lemma \ref{RL}, there are at most $8\sqrt{\delta'_3}t^3$ bad triples in partition $\mathcal{P}^{(1)}$. Thus there exists a cluster $V_0\subseteq V_w$ such that there are at most $8\sqrt{\delta'_3}t^2$ bad triples between $V_0$ and $V_1, \dots , V_{t}$.

Next, we only consider the induced subhypergraph $H'_w:=H_w[\bigcup_{i=0}^{t}V_i]$.
We say that a pair $\{i,j\}\in [t]^{2}$ is {\emph{good}} if $\{V_0,V_i,V_j\}$ is a good triple and there is a good polyad $P^{0jk}_{abc}$ such that $d(H'_w |P^{0jk}_{abc})\geq d$.
Since there are at most $8\sqrt{\delta'_3}t^2$ bad triples containing $V_0$ in $\mathcal{P}^{(1)}$ and $64\delta'_3<\delta^2_3$, there are at most
\[
(8\sqrt{\delta'_3}t^2)\big(\frac{n}{t}\big)^3\leq \frac{2\delta_3}{t}n^3
\]
$3$-edges lying on these bad triples containing $V_0$. Furthermore, if a good triple $\{V_0,V_i,V_j\}$ contains no good polyad $P^{0ij}_{abc}$ with $d(H'_w |P^{0ij}_{abc})\geq d$, then together with the triangle counting lemma, we know that the number of $3$-edges lying on $\{V_0,V_i,V_j\}$ are at most
\[\sum_{\{a,b,c\}\in[\ell]^3}d\cdot \mathcal{K}_3(P^{0ij}_{abc})+\sqrt{\delta'_3}\ell^3\cdot\left(\frac{1}{\ell^3}+3\delta_2\right)\left(\frac{n}{t}\right)^3\leq 2d\left(\frac{n}{t}\right)^3.
\]
We recall that $V_0$ is a set of clone of $w$ and $\delta(H)\ge \alpha n^2$.
Hence, there are at least
\[\sum_{w'\in V_0}\deg_{H_w}(w')\geq |V_0|\cdot\alpha n^2\ge \frac{\alpha(1-\delta'_3)}{t} n^3\]
$3$-edges between $V_0$ and $V$.
Therefore, the number of good pairs in $[t]^2$ is at least
\[
\left(\frac{\alpha(1-\delta'_3)}{t} n^3-\frac{2\delta_3}{t}n^3-2d\binom{t}{2}
\left(\frac{n}{t}\right)^3\right)/\left(\frac{n}{t}\right)^3\geq \frac{\alpha}{2}t^2.
\]

We shall next use Lemma \ref{probability-lemma} to choose a subset $\{i_1,i_2,\dots,i_m\}$ of $[t]$ in which the number of bad triples is zero and the number of good pairs is at least $\frac{\alpha}{4}m^2$. Choose an element $M\in [t]^{m}$ uniformly at random. Let $f(M)$ be the number of bad triples in $M$ and $g(M)$ be the number of good pairs in $M$. Evidently, for any $M',M''\in [t]^{m}$ with $|M'\cap M''|=m-1$, we have
\[
|f(M')-f(M'')|\leq \binom{m-1}{2}\ \ \text{and}\  \  |g(M')-g(M'')|\leq m-1.
\]
Since there are at most $8\sqrt{\delta'_3}t^3<49\sqrt{\delta'_3}\binom{t}{3}$ bad triples and at least $\frac{\alpha}{2}t^2$ good pairs, by the choice of $\delta'_3$, we have
\[
\mathbb{E}(f(M))< \frac{49\sqrt{\delta'_3}\binom{t}{3}\binom{t-3}{m-3}}{\binom{t}{m}}=49\sqrt{\delta'_3}\binom{m}{3}<\frac{1}{3}
\ \ \text{and} \ \
\mathbb{E}(g(M))\geq \frac{\frac{\alpha}{2}t^2\binom{t-2}{m-2}}{\binom{t}{m}}> \alpha\binom{m}{2}.
\]
By Lemma \ref{probability-lemma}, we obtain that
\[\begin{split}
\mathbb{P}\left( f(M)\geq \frac{2}{3}\right) & \leq \mathbb{P}\left( f(M)\geq \mathbb{E}(f(M))+ \frac{1}{3}\right)\\
 &\leq \exp\left(-\frac{2/9}{m(\binom{m-1}{2})^2}\right)
 \leq \exp\left(-\frac{8}{9m^5}\right),
\end{split}\]
and
\[
\begin{split}
\mathbb{P}\left( g(M)\leq \frac{\alpha}{2}\binom{m}{2}\right) & \leq \mathbb{P}\left( g(M)\leq \mathbb{E}(g(M))- \frac{\alpha}{2}\binom{m}{2}\right) \\
& \leq \exp\left(-\frac{1/2\alpha^2(\binom{m}{2})^2}{m(m-1)^2}\right)
 = \exp\left(-\frac{\alpha^2m}{8}\right).
\end{split}
\]
Let
\[
h(m)= \exp\left(-\frac{8}{9m^5}\right)+ \exp\left(-\frac{\alpha^2m}{8}\right).
\]
We get
\[
\lim\limits_{m\to+\infty} h(m)= 1.
\]
On the other hand,
\[
h'(m)= \exp\left(-\frac{8}{9m^5}\right)\cdot\frac{40}{9}m^{-6}
-\exp\left(-\frac{\alpha^2m}{8}\right)\cdot\frac{\alpha^2}{8}>0
\]
for large $m$. Therefore, when $m$ is sufficiently large, the following inequality holds
\[
\mathbb{P}\bigg(f(M)\geq \frac{2}{3}\ \text{or}\ g(M)<\frac{\alpha}{2}\binom{m}{2}\bigg)\leq h(m)<1.
\]
In other words, with positive probability, we can choose an $m$-set $\{i_1,i_2,\dots,i_m\}\subseteq [t]$ such that every triple is good and the number of good pairs is at least $\frac{\alpha}{4}m^2$.

Finally, we construct the desired hypergraph $\hat{H}$. We consider the induced subhypergraph $H''_w:=H'_w[\bigcup_{j=1}^{m}V_{i_j}\cup V_0]$. Let us remove the $3$-edges which lie in a polyad $P$ such that $d(H''_w |P)< d$. Moreover, we also remove the $3$-edges which lie in bad polyads and denote by $\hat{H}$ the resulting $3$-graph after these deletions.
Since for any triple $\{i_j,i_k,i_l\}\subseteq \{0,1,\dots,m\}$, $\{V_{i_j},V_{i_k},V_{i_l}\}$ contains at most $\sqrt{\delta'_3}\ell^3<\delta_3\ell^3$ bad polyads and $d> \delta_3$, we have
\[|E_{\hat{H}}(V_{i_j},V_{i_k},V_{i_l})|\geq |E_{H_w}(V_{i_j},V_{i_k},V_{i_l})|-2d|V_{i_j}||V_{i_k}||V_{i_l}|.\]
Therefore, $\hat{H}$ has all the desired properties.
\end{proof}

To embed hypergraphs of fixed isomorphism type into appropriate, regular and dense polyads of the
partition provided by Lemma \ref{clean}, we shall need an embedding lemma stated in \cite{Reiher_2018}.

\begin{lemma}\label{EL}{\rm (Embedding Lemma~ \cite[Theorem 3.4]{Reiher_2018}).} For every $3$-graph $F$ with vertex set $V(F)=[f]$ and $d_3>0$, there exists $\delta_3>0$ and function $\delta_2:\mathbb{N}\longmapsto(0,1]$ and $N:\mathbb{N}\longmapsto\mathbb{N}$ such that the following holds for every $\ell\in \mathbb{N}$.

Suppose $P=(V_1\dot\cup\cdots\dot\cup V_{f},E(P))$ is a $(\frac{1}{\ell},\delta_2(\ell))$-regular, $f$-partite graph whose vertex classes satisfy $|V_1|=\cdots=|V_{f}|\geq N(\ell)$ and suppose $H$ is an $f$-partite $3$-graph such that for every $3$-edge $\{i,j,k\}\in E(F)$ we have
\begin{itemize}
  \item [{\rm(1)}] $H$ is $\delta_3$-regular w.r.t. the tripartite graph $P[V_i\dot\cup V_j\dot\cup V_k]$ and
  \item [{\rm(2)}] $d(H|P[V_i\dot\cup V_j\dot\cup V_k])\geq d_3$,
\end{itemize}
then $H$ contains a copy of $F$. In fact, there is a monomorphism $q$ from $F$ to $H$ with $q(i)\in V_i$ for all $i\in[f]$.
\end{lemma}

%\begin{lemma}\label{CL}{\rm(Counting Lemma).} Let $F$ be a $3$-graph on $f$ vertex. For every $\xi,d_3>0$, there exist $\delta_3>0$ and function $\delta_2:\mathds{N}\rightarrow(0,1]$ and $N:\mathds{N}\rightarrow\mathds{N}$ such that the following holds for every $\ell\in \mathds{N}$.\\
%Suppose $P=(V_1\dot\cup\cdots\dot\cup V_f,E_P)$ is a $(\delta_2(\ell),\frac{1}{\ell})$-regular, $f$-partite graph whose vertex classes satisfy $|V_1|=\cdots=|V_f|=\hat{n}\geq N(\ell)$ and suppose $H$ is an $f$-partite $3$-graph such that for all edges $ijk$ of $F$ we have
%\begin{itemize}
%  \item [{\rm(1)}] $H$ is $\delta_3$-regular \emph{w. r. t.} the tripartite graph $P[V_i\dot\cup V_j\dot\cup V_k]$ and
%  \item [{\rm(2)}] $d(H|P[V_i\dot\cup V_j\dot\cup V_k])\geq d_3$,
%\end{itemize}
%then $H$ contains at least $(1-\xi)d_3^{e(F)}d_2^{e(\partial F)}\hat{n}^{f}$ copies of $F$.
%\end{lemma}

\subsection{Reduced Hypergraphs.}
In order to use the embedding lemma to show that every $w\in V(H)$ is in one copy of $F$, we need to know the distribution of dense and regular polyads. For this purpose we need to introduce the so-called \textit{reduced hypergraphs}. The terminology below follows \cite[Section~3]{reiher2018some}.

Consider any finite set of indices $I$, suppose that associated with any two distinct
indices $i, j\in I$ we have a finite nonempty set of vertices $\mathcal{P}^{ij}$, and that for distinct pairs of indices the corresponding vertex classes are disjoint. Assume further that for any three distinct indices $i, j, k \in I$ we are given a 3-partite 3-graph $\mathcal{A}^{ijk}$ with vertex classes $\mathcal{P}^{ij}$, $\mathcal{P}^{jk}$ and $\mathcal{P}^{ik}$. Under such circumstances we call the $\binom{|I|}{2}$-partite $3$-graph $\mathcal{A}$ defined by
\[
V(\mathcal{A})=\bigcup_{\{i,j\}\in {I}^{2}}\mathcal{P}^{ij} \ \ \text{and}\ \
E(\mathcal{A})=\bigcup_{\{i,j,k\}\in {I}^{3}}E(\mathcal{A}^{ijk})
\]
a reduced hypergraph. We also refer to $I$ as the \text{index set} of $\mathcal{A}$, to the sets $\mathcal{P}^{ij}$ as the \text{vertex classes} of $\mathcal{A}$, and to the hypergraphs $\mathcal{A}^{ijk}$ as the \text{constituents} of $\mathcal{A}$. For $\eta>0$ such a reduced hypergraph $\mathcal{A}$ is said to be \emph{$\eta$-dense} if
\[
|E(\mathcal{A}^{ijk})|\geq \eta |\mathcal{P}^{ij}||\mathcal{P}^{ik}||\mathcal{P}^{jk}|
\]
holds for every triple $\{i,j,k\}\in [I]^{3}$.

In order to use Lemma \ref{EL} to embed $F$, we need the following two lemmas, where the first one is due to Reiher, R\"{o}dl and Schacht \cite{Reiher2018Hypergraphs}.

\begin{lemma}\label{Rembed}{\rm (\cite[Lemma 3.1]{Reiher2018Hypergraphs}).}
Given $\eta>0$ and $h\in \mathbb{N}$, there exists an integer $q$ such that the following holds. If $\mathcal{A}$ is an $\eta$-dense reduced hypergraph with index set $[q]$, vertex class $\mathcal{P}^{ij}$ and constituents $\mathcal{A}^{ijk}$, then
\begin{itemize}
  \item [{\rm (i)}] there are indices $\lambda(1)<\cdots<\lambda(h)$ in $[q]$ and
  \item [{\rm (ii)}] for each pair $1\leq r<s\leq h$ there are three vertices $P^{\lambda(r)\lambda(s)}_{{\rm{\color{red} red}}}$, $P^{\lambda(r)\lambda(s)}_{{\rm {\color{blue}blue}}}$ and $P^{\lambda(r)\lambda(s)}_{{\rm {\color{green}green}}}$ in $\mathcal{P}^{\lambda(r)\lambda(s)}$ such that for every triple indices $1\leq r<s<t\leq h$ the three vertices $P^{\lambda(r)\lambda(s)}_{{\rm{\color{red} red}}}$, $P^{\lambda(r)\lambda(t)}_{{\rm {\color{blue}blue}}}$ and $P^{\lambda(s)\lambda(t)}_{{\rm {\color{green}green}}}$ form a hyperedge in $\mathcal{A}^{\lambda(r)\lambda(s)\lambda(t)}$.
\end{itemize}
\end{lemma}

\begin{lemma}\label{biembed}
Given $\eta>0$ and $h_1,h_2\in \mathbb{N}$, there exist integers $q_1$ and $q_2$ such that the following holds. If $\mathcal{A}$ is an $\eta$-dense reduced hypergraph with index set $[q_1+q_2]$, vertex class $\mathcal{P}^{ij}$ and constituents $\mathcal{A}^{ijk}$, then there are indices $\lambda(1)<\cdots<\lambda(h_1)$ in $[q_1]$ and $\sigma(1)<\cdots<\sigma(h_2)$ in $[q_1+1,q_1+q_2]$ satisfying that
\begin{itemize}
  \item [{\rm (i)}] for each pair $1\leq r<s\leq h_1$ there is a vertex $P^{\lambda(r)\lambda(s)}_{{\rm {\color{red}red}}}$;
  \item [{\rm (ii)}] for any $1\leq r\leq h_1$ and any $1\leq t\leq h_2$ there are two vertices $P^{\lambda(r)\sigma(t)}_{{\rm {\color{blue}blue}}}$ and $P^{\lambda(r)\sigma(t)}_{{\rm {\color{green}green}}}$;
  \item [{\rm (iii)}] for every triple indices $1\leq r<s\leq h_1$ and $1\leq t\leq h_2$ the three vertices $P^{\lambda(r)\lambda(s)}_{{\rm {\color{red}red}}}$, $P^{\lambda(r)\sigma(t)}_{{\rm {\color{blue}blue}}}$ and $P^{\lambda(s)\sigma(t)}_{{\rm {\color{green}green}}}$ form a hyperedge in $\mathcal{A}^{\lambda(r)\lambda(s)\sigma(t)}$.
\end{itemize}
\end{lemma}
\begin{proof}
The proof here follows the idea in the proof of Lemma \ref{Rembed}, arguing by greedily filtering the useful clusters by the pigonhole principle.
Choose $q_1$ and $q_2$ such that
\[\max\{h_1,h_2,1/\eta\}\ll q_1\ll q_2.\]
Our proof contains the following three stages.

\noindent{\bf Choosing $\bm{P^{rs}_{{\rm {\color{red}red}}}}$.} For any three indices $r,s,t$ where $1\leq r<s\leq q_1$ and $1\leq t\leq q_2$, we denote the degree of a vertex $P\in \mathcal{P}^{rs}$ in $\mathcal{A}^{rst}$ by $d_t(P)$, that is, $d_t(P)=|\{(Q,R)\in \mathcal{P}^{rt}\times \mathcal{P}^{st}:\{P,Q,R\}\in \mathcal{A}^{rst}\}|$. We set
\[\mathcal{P}^{rs}_{t,{\rm {\color{red}red}}}=\left\{P\in \mathcal{P}^{rs}:d_t(P)\geq\frac{\eta}{2}\cdot|\mathcal{P}^{rt}||\mathcal{P}^{st}|\right\}.\]
Since $\mathcal{A}$ is $\eta$-dense, we have
\begin{align}
\eta|\mathcal{P}^{rs}||\mathcal{P}^{rt}||\mathcal{P}^{st}|\leq&|E(\mathcal{A}^{rst})|=\sum_{P\in\mathcal{P}^{rs}}d_t(P)=\sum_{P\in\mathcal{P}^{rs}\setminus\mathcal{P}^{rs}_{t,{\rm {\color{red}red}}}}d_t(P)+\sum_{P\in\mathcal{P}^{rs}_{t,{\rm {\color{red}red}}}}d_t(P)\nonumber\\
\leq& \frac{\eta}{2}\cdot|\mathcal{P}^{rs}||\mathcal{P}^{rt}||\mathcal{P}^{st}|+|\mathcal{P}^{rs}_{t,{\rm {\color{red}red}}}||\mathcal{P}^{rt}||\mathcal{P}^{st}|.\nonumber
\end{align}
Therefore, $|\mathcal{P}^{rs}_{t,{\rm {\color{red}red}}}|\geq\frac{\eta}{2}\cdot|\mathcal{P}^{rs}|$.

We determine $P^{rs}_{{\rm {\color{red}red}}}$ for each pair $1\leq r<s\leq q_1$ increasingly with respect to the lexicographical order. Set $\mathcal{I}_0:=[q_1+q_2]\setminus[q_1]$. For the $i$th pair $1\leq r<s\leq q_1$, by double counting, there is an index set $\mathcal{I}_i\subseteq \mathcal{I}_{i-1}$ of cardinality at least $\frac{\eta |\mathcal{I}_{i-1}|}{2}$ and there is an element $P^{rs}_{{\rm {\color{red}red}}}\in \mathcal{P}^{rs}$ satisfying
\[P^{rs}_{{\rm {\color{red}red}}}\in\bigcap_{t\in \mathcal{I}_i}\mathcal{P}^{rs}_{t,{\rm {\color{red}red}}}.\]
Therefore, after $\binom{q_1}{2}$ steps, we eventually choose $P^{rs}_{{\rm {\color{red}red}}}$ for each pair $1\leq r<s\leq q_1$ and obtain an index set $\mathcal{I}:=\mathcal{I}_{\binom{q_1}{2}}\subseteq [q_1+q_2]\setminus[q_1]$ such that
\[P^{rs}_{{\rm {\color{red}red}}}\in\bigcap_{t\in \mathcal{I}}\mathcal{P}^{rs}_{t,{\rm {\color{red}red}}}\]
for all pairs $1\leq r<s\leq q_1$. One remark here is that we may assume $|\mathcal{I}|=h_2$ since otherwise we can proceed with a subset of $\mathcal{I}$ of cardinality $h_2$.

\noindent{\bf Choosing $\bm{P^{rt}_{{\rm {\color{blue}blue}}}}$.} We next choose an index set $\mathcal{J}\subseteq [q_1]$ and determine $P_{{\rm {\color{blue}blue}}}^{rt}$ in $\mathcal{P}^{rt}$ for every $(r,t)\in \mathcal{J}\times \mathcal{I}$. For any pair $1\leq r<s\leq q_1$ and any $t\in \mathcal{I}$, let us define
\[\mathcal{P}^{rt}_{s,{\rm {\color{blue}blue}}}=\left\{Q\in \mathcal{P}^{rt}:d(P^{rs}_{{\rm {\color{red}red}}},Q)\geq \frac{\eta}{4}\cdot|\mathcal{P}^{st}|\right\},\]
where $d(P^{rs}_{{\rm {\color{red}red}}},Q)=|\{R\in \mathcal{P}^{st}:\{P^{rs}_{{\rm {\color{red}red}}},Q,R\}\in \mathcal{A}^{rst}\}|$. Since
\[d_t(P^{rs}_{{\rm {\color{red}red}}})=\sum_{Q\in \mathcal{P}^{rt}}d(P^{rs}_{{\rm {\color{red}red}}},Q)\geq\frac{\eta}{2}\cdot|\mathcal{P}^{rt}||\mathcal{P}^{st}|,\]
the same calculation as above discloses $\mathcal{P}^{rt}_{s,{\rm {\color{blue}blue}}}\geq\frac{\eta}{4}\cdot|\mathcal{P}^{st}|$. Set $\mathcal{J}_0:=[q_1]$ and suppose we proceed with $\mathcal{J}_i$ now. Let $r$ be the smallest index in $\mathcal{J}_i$ such that $P_{{\rm {\color{blue}blue}}}^{rt}$ has not been determined for some $t\in \mathcal{I}$. Denote by $\mathcal{J}'_i$ the set of all elements in $\mathcal{J}_i$ that are larger than $r$. By double counting, there exists a subset $\mathcal{J}''_i$ of $\mathcal{J}'_i$ of cardinality at least $\frac{\eta|\mathcal{J}'_i|}{4}$ and there is an element $Q$ in $\mathcal{P}^{rt}$ such that
\[Q\in\bigcap_{s\in \mathcal{J}''_i}\mathcal{P}^{rt}_{s,{\rm {\color{blue}blue}}}.\]
\noindent Let $P_{{\rm {\color{blue}blue}}}^{rt}=Q$ and set $\mathcal{J}_{i+1}=(\mathcal{J}_i\setminus\mathcal{J}'_i)\cup\mathcal{J}''_i$. Finally we get an index set $\mathcal{J}\subseteq [q_1]$ and choose $P_{{\rm {\color{blue}blue}}}^{rt}$ for any $(r,t)\in \mathcal{J}\times \mathcal{I}$ such that
\[P_{{\rm {\color{blue}blue}}}^{rt}\in\bigcap_{s\in \mathcal{J},s>r}\mathcal{P}^{rt}_{s,{\rm {\color{blue}blue}}}.\]

\noindent{\bf Choosing $\bm {P^{st}_{{\rm {\color{green}green}}}}$.} We further shrink the index set $\mathcal{J}$ to a new one $\mathcal{L}$ and determine $P_{{\rm {\color{green}green}}}^{st}$ in $\mathcal{P}^{st}$ for every $(s,t)\in \mathcal{L}\times \mathcal{I}$. When choosing $P_{{\rm {\color{green}green}}}^{st}$, we shall ensure that for any pairs $r<s$ in $\mathcal{L}$ and any $t\in\mathcal{I}$, $\{P_{{\rm {\color{red}red}}}^{rs},P_{{\rm {\color{blue}blue}}}^{rt},P_{{\rm {\color{green}green}}}^{st}\}$ is a hyperedge in $\mathcal{A}^{rst}$, and thus Property {\rm(iii)} will be satisfied. For this purpose, we set
\[
\mathcal{P}^{st}_{r,{\rm {\color{green}green}}}=\{R\in \mathcal{P}^{st} : \{P_{{\rm {\color{red}red}}}^{rs},P_{{\rm {\color{blue}blue}}}^{rt},R\} \in E(\mathcal{A}^{rst})\}
\]
for any pair $r<s$ in $\mathcal{J}$ and any $t\in\mathcal{I}$. Then $|\mathcal{P}^{st}_{r,{\rm {\color{green}green}}}|\geq \frac{\eta}{4}\cdot|\mathcal{P}^{st}|$ follows from our previous choices.

 Set $\mathcal{L}_0:=\mathcal{J}$ and suppose we proceed with $\mathcal{L}_i$ now. Consider the largest index $s\in \mathcal{L}_i$ for which there exists an index $t\in \mathcal{I}$ with $P^{st}_{{\rm {\color{green}green}}}$ undetermined. Denote by $\mathcal{L}'_i$ the set of all smaller elements than $s$ in $\mathcal{L}_i$. By double counting, there exists a subset $\mathcal{L}''_i$ of $\mathcal{L}'_i$ of cardinality at least $\frac{\eta|\mathcal{L}'_i|}{4}$ and there is an element $R$ in $\mathcal{P}^{st}$ such that
\[R\in\bigcap_{r\in \mathcal{L}''_i}\mathcal{P}^{st}_{r,{\rm {\color{green}green}}}.\]
\noindent Let $P_{{\rm {\color{green}green}}}^{st}=R$ and set $\mathcal{L}_{i+1}=(\mathcal{L}_i\setminus\mathcal{L}'_i)\cup\mathcal{L}''_i$. Finally we obtain an index set $\mathcal{L}\subseteq \mathcal{J}$ and choose $P_{{\rm {\color{green}green}}}^{st}$ for any $(s,t)\in \mathcal{L}\times \mathcal{I}$ such that
\[P_{{\rm {\color{green}green}}}^{st}\in\bigcap_{r\in \mathcal{L},r<s}\mathcal{P}^{st}_{r,{\rm {\color{green}green}}}.\]

Without loss of generality, we may assume $\mathcal{L}=[h_1]$ and $\mathcal{I}=[q_1+h_2]\setminus[q_1]$. Now it is plain that the indices $\lambda(r)=r$ for $r\in[h_1]$ and $\sigma(t)=q_1+t$ for $t\in [h_2]$ are as desired.
\end{proof}

\noindent{\textbf{Remark.}} In the subsequent proof we also need another statement of Lemma~ \ref{biembed} as follows.

\emph{
Given $\eta>0$ and $h_1,h_2\in \mathbb{N}$, there exist integers $q_1$ and $q_2$ such that the following holds. If $\mathcal{A}$ is an $\eta$-dense reduced hypergraph with index set $[q_1+q_2]$, vertex class $\mathcal{P}^{ij}$ and constituents $\mathcal{A}^{ijk}$, then there are indices $\lambda(1)<\cdots<\lambda(h_1)$ in $[q_1]$ and $\sigma(1)<\cdots<\sigma(h_2)$ in $[q_1+1,q_1+q_2]$ satisfying that
\begin{itemize}
  \item [{\rm (i)}] for each pair $1\leq s<t\leq h_2$ there is a vertex $P^{\sigma (s)\sigma (t)}_{{\rm {\color{green}green}}}$;
  \item [{\rm (ii)}] for any $1\leq r\leq h_1$ and any $1\leq t\leq h_2$ there are two vertices $P^{\lambda(r)\sigma(t)}_{{\rm {\color{red}red}}}$ and $P^{\lambda(r)\sigma(t)}_{{\rm {\color{blue}blue}}}$;
  \item [{\rm (iii)}] for every triple indices $1\leq r\leq h_1$ and $1\leq s<t\leq h_2$ the three vertices $P^{\lambda(r)\sigma (s)}_{{\rm {\color{red}red}}}$, $P^{\lambda(r)\sigma(t)}_{{\rm {\color{blue}blue}}}$ and $P^{\sigma (s)\sigma(t)}_{{\rm {\color{green}green}}}$ form a hyperedge in $\mathcal{A}^{\lambda(r)\sigma (s)\sigma(t)}$.
\end{itemize}
}

%\subsection{Proof of the backward implication of Theorem \ref{equal}.}

\begin{proof}[Proof of the backward implication of Theorem \ref{equal}]Suppose a $3$-graph $F$ on $f$ vertices satisfies properties \rm{(i)} and \rm{(ii)} in Theorem \ref{equal} and $p,\alpha>0$ are given. Let $\mathcal{P}=\{X,Y,\{v^*\}\}$ be a partition of $V(F)$ as in Theorem \ref{equal} and let $\{v_1,\dots,v_f\}$ be an enumeration of $V(F)$ as in Lemma~\ref{alter}. By Theorem \ref{k-graph}, we only need to check that $F\in \cover_3$, i.e., there exist constants $n_0,\mu>0$ such that for $n\ge n_0$ every vertex $w$ in an $(n,p,\mu,\alpha)$ $3$-graph $H$ is contained in a copy of $F$.

Here is a brief outline before we proceed the formal proof.
We first apply Lemma \ref{clean} to $H$ and obtain a subhypergraph $\hat H$ and a cluster $V_0$ consisting of clones of $w$.
Then our goal is to apply the embedding lemma to an appropriate regular and dense $(f,3)$-complex with a cluster being $V_0$.
By the minimum degree condition, we can choose a large candidate set of clusters, labeled by $[2C]$, with a bipartite structure $K_{C, C}$, that is, each pair with one from $[C]$ and the other from $[C+1, 2C]$ forms a regular triple together with $V_0$.
Then we greedily filter the candidate set of clusters by Lemmas~\ref{Rembed} and~\ref{biembed} and obtain an appropriate $(f,3)$-complex to which we can apply Lemma~\ref{EL}.
In particular, Lemma~\ref{Rembed} is used to filter clusters for embedding edges inside each of the partition and Lemma~\ref{biembed} is used to filter clusters for embedding edges between the bipartition.
%In particular, Lemma~\ref{Rembed} is used to filter the clusters inside each of the bipartition of clusters and Lemma~\ref{biembed} is used to filter the clusters between the bipartition.

\medskip
Suppose $f_1=|X|$ and $f_2=|Y|$. Set $d_3=\min\{\frac{p}{9},\frac{\alpha}{9}\}$ and $\eta=\frac{d_3}{2}$. Let %\[\max\{f_1,f_2,\lceil2/d_3\rceil\}\ll m_1\ll m_2\ll M_2\ll M_1\ll C\]
\[1/C\ll 1/M_1\ll 1/M_2\ll 1/m_2\ll 1/m_1\ll \max\{f_1,f_2,\lceil2/d_3\rceil\}\]
and $m_1,m_2,M_1,M_2,C$ are all integers. Choose a sufficiently large integer $m$ such that any graph on $m$ vertices with at least $\frac{\alpha m^2}{4}$ edges contains a complete bipartite graph $K_{C,C}$ and $m$ satisfies the requirement of Lemma~\ref{clean}.
Plugging $F$ and $d_3$ into Lemma \ref{EL} we get a constant $\delta_3>0$, a function $\delta_2:\mathbb{N}\longmapsto(0,1]$ and a function $N:\mathbb{N}\longmapsto\mathbb{N}$. Obviously we may assume that $\delta_3\leq \min\{\frac{p}{8},\frac{\alpha}{8}\}$, $\delta_2(\ell)\ll \ell^{-1}$ and $N$ is increasing.
Choose $\mu>0$ to be sufficiently small, and an integer $n_0$ to be sufficiently large.

Now let $H=(V,E)$ be an $(n,p,\mu,\alpha)$ $3$-graph with $n\ge n_0$.
Given a vertex $w\in V$, in order to show that $w$ lies in a copy of $F$, we first apply Lemma \ref{clean} with $d_3,\delta_3,\alpha,m$ and $\delta_2(\cdot)$ to $H_w$ to obtain
\begin{itemize}
  \item a subhypergraph $\hat{H}=(\hat{V},\hat{E})\subseteq H_w$;
  \item a vertex partition $V_0\dot\cup V_1\dot\cup\cdots\dot\cup V_m=\hat{V}$ with $V_0\subseteq V_w$ and $\bigcup_{i=1}^{m}V_i\subseteq V$;
  \item an integer $\ell\leq T_0$;
  \item partitions $\mathcal{P}^{ij}=\{P^{ij}_{a}=(V_i\dot\cup V_j,E^{ij}_{a}):1\leq a\leq \ell\}$ of $K(V_i,V_j)$ for all integers $i,j$ with $0\leq i<j\leq m$
\end{itemize}
satisfying Properties {\rm (i)}--{\rm(v)} from Lemma \ref{clean}.
In particular, by Property {\rm(v)}, there is a subset of $[m]$ of cardinality $2C$, say $[2C]$, such that for all $i\in [C]$ and $j\in [2C]\setminus[C]$, $d(\hat{H}|P^{0ij}_{abc})\geq d_3$ holds for at least one triple $\{a,b,c\}\in [\ell]^3$. In this situation, our reduced hypergraph $\mathcal{A}$ has index set $\{0\}\cup[2C]$, vertex class $\mathcal{P}^{ij}$ and a triple $\{P^{ij}_a,P^{ik}_b,P^{jk}_c\}$ is defined as an edge of the constituent $\mathcal{A}^{ijk}$ if and only if $d(\hat{H}|P^{ijk}_{abc})\geq d_3$. As shown in the proof of Theorem 1.2 in \cite [Section 3]{Reiher2018Hypergraphs}, our reduced hypergraph $\mathcal{A}$ is $\eta$-dense restricted to the index set $[2C]$.

In order to use Lemma \ref{EL} to find a copy of $F$ containing $w$, we need to find a regular $f$-partite graph (with a vertex class being $V_0$) for $F$.
For this purpose, we first apply Lemma~\ref{Rembed} with the index set $[C]$ to find an index subset $\mathcal{I}_1$ of $[C]$ with $|\mathcal{I}_1|=M_1$ such that Property {\rm(ii)}  from Lemma~\ref{Rembed} is satisfied.
Then we apply Lemma \ref{Rembed} again with the index set $[2C]\setminus[C]$ to find an index subset $\mathcal{I}_2$ of cardinality $M_2$ in $[2C]\setminus[C]$ satisfying Property {\rm(ii)} from Lemma~\ref{Rembed}.
Next we apply Lemma~\ref{biembed} (statement in the remark) with $h_1:=m_1$, $h_2:=m_2$ and the set $\mathcal{I}_1\cup\mathcal{I}_2$ to find a subset $\mathcal{J}_1$ of $\mathcal{I}_1$ of cardinality $m_1$ and a subset $\mathcal{J}_2$ of $\mathcal{I}_2$ of cardinality $m_2$ such that Properties {\rm(i)}--{\rm(iii)} from Lemma \ref{biembed} are satisfied.
Finally, we apply Lemma \ref{biembed} again with $h_1:=f_1$, $h_2:=f_2$ and the index set $\mathcal{J}_1\cup\mathcal{J}_2$ to find a subset $\mathcal{X}$ of $\mathcal{J}_1$ of cardinality $f_1$ and a subset $\mathcal{Y}$ of $\mathcal{J}_2$ of cardinality $f_2$ satisfying Properties {\rm(i)}--{\rm(iii)} from Lemma \ref{biembed}. Without loss of generality, we may assume that $\lambda(1)=0$, $\mathcal{X}=\{\lambda(2),\dots,\lambda(f_1+1)\}$ and $\mathcal{Y}=\{\lambda(f_1+2),\dots,\lambda(f)\}$ with $\lambda(1)<\dots<\lambda(f)$.

To conclude the proof, we show that there exists a $w_i\in V_{\lambda(i)}$ for each $i\in [f]$ which altogether support a copy of $F$ in $\hat{H}$ (note that $w_1$ actually equals to $w$). We construct an auxiliary $3$-graph $\hat{F}$ from $F$ by deleting the vertex $v^*$ and adding $|N_F(v^*)|$ new vertices such that each of them forms a hyperedge of $\hat{F}$ with a distinct pair in $N_F(v^*)$.
For any $\{v_i,v_j,v_k\}\in \cup_{i=0}^{3}F_i$ with $i<j<k$, by our choices of $\mathcal{X}$ and $\mathcal{Y}$, we have $d(\hat{H}|P^{\lambda(i)\lambda(j)\lambda(k)}_{{\rm {\color{red}red},{\color{blue}blue},{\color{green}green}}})\geq d_3$.
For any new vertex $u\in V(\hat{F})$ and $\{u,v_i,v_j\}\in \hat{F}$ with $i<j$, since $N_F(v^*)\subseteq X\times Y$, we know that $v_i\in X$ and $v_j\in Y$.
By the choices of $\mathcal{X}$ and $\mathcal{Y}$, we conclude that $\{\lambda(i),\lambda(j)\}$ is a good pair for $V_0$, which means that $d(\hat{H}\mid P^{0\lambda(i)\lambda(j)}_{abc})\geq d_3$ for some good polyad $P^{0\lambda(i)\lambda(j)}_{abc}$. Therefore, since $\partial F_{v^*}$, $\partial F_0$, $\partial F_1$, $\partial F_2$ and $\partial F_3$ are pairwise disjoint, by Lemma \ref{EL}, we can embed $\hat{F}$ such that $v_i$ is embedded into $V_{\lambda(i)}$ for $2\leq i\leq f$ and all new vertices are embedded into $V_0$. Since every vertex from $V_0$ is a copy of $w$, we have embedded $F$ into $H$ with mapping $v^*$ to $w$.
\end{proof}

\section{Avoiding $F$-factors}

In this section, we prove the following lemma and use it to derive the forward implications of Theorems~\ref{k-partite} and~\ref{equal}.
%whose {\rm (ii)} implies the forward implication of Theorem~\ref{k-partite} and whose {\rm (i)} and {\rm (ii)} prove the forward implication of Theorem~\ref{equal}.

%Given a $k$-graph $H$ and a subset $S\subseteq V(H)$ with $|S|=s$, the \emph{degree} of $S$, denoted by $\deg_H(S)$ or $\deg(S)$, is the number of edges of $H$ containing $S$ as a subset. Let $N_H(S)$ or $N(S)$ denote the neighbor set of $S$, i.e., $N_H(S)=\{ S' \in [V(H)]^{k-s} : S'\cup S\in E(H)\}$.
%The \emph{minimum $s$-degree} $\delta_s(H)$ is the minimum of $\deg(S)$ over all $s$-subsets $S$ of $V(H)$.

%In particular, we call the minimum $1$-degree and the minimum $(k-1)$-degree of $H$ as the\emph{ minimum vertex degree} and \emph{minimum codegree} of $H$ respectively. If $S=\{v\}$ is a singleton then we write that $\deg_H(v)$ and $N_H(v)$ simply.

\begin{lemma}\label{k-construction}
For $k\geq 3$, if a $k$-graph $F\in \factor_k$, then there exists a vertex $v^*$ and a partition $\mathcal{P}=\{X_1,X_2,\dots ,X_{k-1},\{v^*\}\}$ of $V(F)$ such that $N_F(v^*)\subseteq X_1\times \cdots \times X_{k-1}$ and for every two edges $e, e'\in E(F)$, if $|e\cap e'| \ge 2$, then ${\bf i} _{\mathcal{P}}(e)={\bf i} _{\mathcal{P}}(e')$.
%\begin{itemize}
%  \item[{\rm (i)}] $\pi_{\points}(F)=0$;
%  \item[{\rm (ii)}] there exists a partition $\mathcal{P}=\{X_1,X_2,\dots ,X_{k-1},\{v^*\}\}$ of $V(F)$ such that $N(v^*)$ is a $(k-1)$-partite $(k-1)$-graph under partition $\mathcal{P}$, and for every two edges $e, e'\in E(F)$, if $|e\cap e'| \ge 2$, then ${\bf i} _{\mathcal{P},F}(e)=\textbf{i} _{\mathcal{P},F}(e')$.
%\end{itemize}
\end{lemma}

To see the forward implications of Theorems~\ref{k-partite} and~\ref{equal}, fix $F\in \factor_k$ and take $v^*$ and the partition $\mathcal{P}=\{X_1, \dots, X_{k-1}, \{v^*\}\}$ given by Lemma~\ref{k-construction}.
%Note that for two edges $e$ and $e'$ of $F$ with $|e\cap e'|\ge 2$, then either both of them contain $v^*$ or neither of them contains $v^*$,
Fix two edges $e, e'$ such that $v^*\in e$ and $v^*\notin e'$, because ${\bf i} _{\mathcal{P}}(e)\neq {\bf i} _{\mathcal{P}}(e')$, we infer that $|e\cap e'|\le 1$, implying Theorem~\ref{k-partite}.
For Theorem~\ref{equal}, note that Theorem~\ref{equal}~\ref{item:i} has been explained in the introduction.
For~\ref{item:ii}, applying Lemma~\ref{k-construction} with $k=3$, we have $N_F(v^*)\subseteq X_1\times X_2$ and for any $x\in X_1$ and $y\in X_2$, if $N(x)\cap N(y)\neq \emptyset$, then there are two edges $e, e'\in E(F)$ with $x\in e$, $y\in e'$ and $|e\cap e'|=2$. By the property of $\mathcal P$,  we must have ${\bf i} _{\mathcal{P}}(e)={\bf i} _{\mathcal{P}}(e')$, which is impossible as  $e\setminus\{x\}=e \setminus \{y\}$ and $x\in X_1$ and $y\in X_2$. Similarly, we also have $N(x)\cap N(v^*)=\emptyset$ and $N(y)\cap N(v^*)=\emptyset$.

At the first sight it seems that to generalize Theorem~\ref{equal}~\ref{item:ii} to $k$-graphs, we should extend the partition to $k$ parts (still with a specified vertex $v^*$) and require that vertices from different parts do not share any links.
We remark here that the property of $\mathcal P$ from Lemma~\ref{k-construction} is stronger than this ``direct'' generalization.

\begin{proof}[Proof of Lemma~\ref{k-construction}]
Given $k\geq 3$, let $F$ be a $k$-graph with $F\in \factor_k$ and $f:=v(F)$.
%As we have discussed in the introduction, $F$ satisfies Property {\rm (i)}.
%We will prove that $F$ satisfies Properties {\rm (ii)} using the following construction.

For $n\in f\mathbb{N}$, define a probability distribution $H(n)$ on $k$-graphs of order $n$ as follows.
Let $K:=K_n$ be the complete graph on $n$ vertices, and $\mathcal{Q}=\{V_1,\dots, V_{k-1},\{z\}\}$ be a partition of $V(K)$ such that $\textbf{i} _{\mathcal{Q}}(V(K))=(n_1,\dots, n_{k-1},1)$ with $n_i\ge n/k$.
%For each $k$-set $e\subset V(K)$, let $I=\{\textbf{i} _{\mathcal{P},K}(e): z\notin e \text {\ or}  \  \textbf{i} _{\mathcal{P},K}(e)=(1,\dots, 1)\}$, then $|I|=\binom{2k-2}{k}+1$. For convenience, let $f(k):=|I|$ and $I=\{\textbf{i}_1,\textbf{i}_2,\dots ,\textbf{i}_{f(k)}\}$.
Write $r:=r(k)=\binom{2k-2}{k}$.
Let $I=\{\textbf{i}_1,\dots, \textbf{i}_{r}\}$ be an enumeration of the $k$-vectors whose coordinates are nonnegative and sum to $k$ on $\mathbb{Z}^{k}$ with last digit 0 and write $\textbf{i}_0:=(1,\dots, 1)\in \mathbb Z^k$.
Define a random $(r+1)$-coloring $\phi: E(K) \longmapsto \{c_0,c_1,\dots ,c_{r}\}$ where each color is assigned to an edge of $K$ with probability $1/r$ independently.
We say a clique in $K$ is {\emph {$c_j$-colored}} if all edges of this clique have color $c_j$, where $0\le j\le r$.
Let now the vertex set of $H(n)$ be $V(K)$ and include a $k$-set $e\subset V(K)$ in $E(H(n))$ if $\textbf{i} _{\mathcal{Q}}(e)=\textbf{i}_{j}$ with $0\le j\le r$ and $K[e]$ is a $c_{j}$-colored clique.
%
%Let now the vertex set of $H(n)$ be $V(K)$,
%then $\mathcal{P}=\mathcal{P'}\cup \{z\}$ is a partition of $V(H(n))$ with $\textbf{i} _{\mathcal{P},H(n)}=(n_1,\dots, n_{k-1},1)$.
%make a $k$-set $e\subset V(H(n))$ into an edge of $H(n)$ as follows.
%If $\textbf{i} _{\mathcal{P},H(n)}(e)\big |_\mathcal{P'}=\textbf{i}_{\ell}$ with $\ell \in [f(k)]$, then add $e$ to $H(n)$ if $K[e]$ is a $c_{\ell}$-colored clique; if $\textbf{i} _{\mathcal{P},H(n)}(e)\big |_\mathcal{P'}=(1,\dots,1)$, then add $e$ to $H(n)$ if $K[e]$ is a $c_{f(k)+1}$-colored clique.
%If $\textbf{i} _{\mathcal{P},H(n)}(e)=\textbf{i}_{\ell}$ with $\ell \in [f(k)]$, then add $e$ to $H(n)$ if $K[e]$ is a $c_{\ell}$-colored clique.
Therefore, for each $k$-set $e$ not containing $z$, $e\in E(H(n))$ with probability $p^*=(r+1)^{-\binom{k}{2}}$, and for each vertex $v\in V(H(n))$, the expectation of its degree is
\[
\deg_{H(n)} (v)\ge \min \left\{p^*\binom{n-2}{k-1}, p^*\prod^{k-1}_{i=1}n_i \right\} = p^*\prod^{k-1}_{i=1}n_i \ge \frac{p^*}{k^{k-1}}n^{k-1}.
\]
By concentration inequalities (e.g. Janson's inequality) and the union bound, for sufficiently large $n$,
there exists a $k$-graph $H\in H(n)$ such that $H$ is $(p^*,\mu)$-dense with
$\delta _1(H)\ge (p^*/k^{k}) n^{k-1}$. Moreover, we have
\begin{itemize}
\item[$(\dagger)$] $N_H(z)\subseteq V_1\times \cdots \times V_{k-1}$ and for every two edges $e, e'\in E(H)$,
if $|e\cap e'| \ge 2$, then $K[e]$ and $K[e']$ must have the same color, which implies that $\textbf{i} _{\mathcal{Q}}(e)=\textbf{i} _{\mathcal{Q}}(e')$.
\end{itemize}

%with $\textbf{i}_{\mathcal{P},H}(e)\neq \textbf{i} _{\mathcal{P},H}(e')$, $\textbf{i}_{\mathcal{P},H}(e)\big |_\mathcal{P'}\neq \textbf{i} _{\mathcal{P},H}(e')\big |_\mathcal{P'}$, which implies that $|e\cap e'|<2$.

Next we show the lemma.
Note that $H$ has an $F$-factor because $F\in \factor_k$, $n\in f\mathbb N$, and that $H$ is an $(n,p^*,\mu,p^*/k^{k})$ $k$-graph.
Therefore, $z$ is contained in a copy $F'$ of $F$ in $H$.
Now, as $F'$ is a subgraph of $H$, the property $(\dagger)$ is passed from $H$ to $F'$ with the partition $\mathcal Q=\{V_1,\dots, V_{k-1}, \{z\}\}$ replaced by $\{W_1,\dots, W_{k-1}, \{z\}\}$, where $W_i=V_i\cap V(F')$, and further passed to $F$ (which creates a partition of $V(F)$ as in the statement of the lemma) as $F'$ is isomorphic to $F$.
\end{proof}

\section{Concluding remarks}
In this paper, we studied Problem~\ref{prob} and reduced it to the $F$-cover problem.
Furthermore, we completely solved this problem for $k=3$ and gave an explicit characterization of such $F$ in Problem~\ref{prob}.
For $k\ge 4$, we believe the conditions in Lemma~\ref{k-construction} together with $\pi_{\points}(F)=0$ give a characterization for $F$ in Problem~\ref{prob} and put forward the following conjecture.

\begin{conj}\label{k-equal}
For $k\ge 4$, a $k$-graph $F\in \factor_k$ if and only if it satisfies the following properties.
\begin{itemize}
  \item[$(i)$] $\pi_{\points}(F)=0$;
  \item[$(ii)$] there exists a vertex $v^*$ and a partition $\mathcal{P}=\{X_1,X_2,\dots ,X_{k-1},\{v^*\}\}$ of $V(F)$ such that $N_F(v^*)\subseteq X_1\times \cdots \times X_{k-1}$ and for every two edges $e, e'\in E(F)$, if $|e\cap e'| \ge 2$, then ${\bf i} _{\mathcal{P}}(e)={\bf i} _{\mathcal{P}}(e')$.
\end{itemize}
\end{conj}
%Clearly, Conjecture~\ref{k-equal} and Theorem~\ref{equal} are equivalent for $k=3$.

%Conjecture~\ref{k-equal} holds for all $k$-graphs $F$ that are $k$-partite or linear.

\subsection{Other quasi-randomness conditions}
It is natural to consider the $F$-factor problem under other (stronger) quasi-randomness and degree conditions.
In particular, Lenz and Mubayi~\cite{LenzPacking2} gave a notion of quasi-randomness which is sufficient to force $F$-factor for arbitrary $k$-graph $F$.
Our discussion below will be based on the notions of quasi-randomness from \cite{Reiher_2017}.

Given a finite set $V$ and a set $S\subseteq [k]$, let $V^S$ be the set of all functions from $S$ to $V$. For convenience, we identify the Cartesian power $V^k$ with $V^{[k]}$ by regarding any $k$-tuple $\vec{v} = (v_1,\dots ,v_k)$ as being the function $i\mapsto v_i$. In this way, the natural projection from $V^k$ to $V^S$ becomes the restriction $\vec{v}\mapsto \vec{v}\mid S$ and the preimage of any set $G_S\subseteq V^S$ is denoted by
\[
\mathcal{K}_k(G_S) = \{\vec{v}\in V^k:(\vec{v}\mid S)\in G_S \}
\]
One may think of $G_S\subseteq V^S$ as a directed hypergraph (where vertices in the directed hyperedges are also allowed to repeat).

More generally, for a subset $\mathscr{S}\subseteq \mathcal{P}([k])$ of the power set of $[k]$ and a family $\mathscr{G} = \{G_S:S\in \mathscr{S}\}$ with $G_S \subseteq V^S$ for all $S\in \mathscr{S}$, let
\[
\mathcal{K}_k(\mathscr{G})= \bigcap \limits_{S\in \mathscr{S}} \mathcal{K}_k(G_S).
\]
If $H = (V,E)$ is a $k$-uniform hypergraph on $V$, then $e_H(\mathscr{G})$ denotes the cardinality of the set
\[
E_H(\mathscr{G})=\big \{(v_1,\dots ,v_k)\in \mathcal{K}_k(\mathscr{G}):\{ v_1,\dots ,v_k\}\in E\big \}.
\]

\begin{defi}{\rm (\cite[Definition 2.1]{Reiher_2017}).}
Let real numbers $p\in [0,1]$ and $\mu >0$, a $k$-graph $H=(V,E)$ and a set $\mathscr{S}\subseteq \mathcal{P}([k])$ be given. We say that $H$ is $(p,\mu ,\mathscr{S})$\emph{-dense} provided that
\[
e_H(\mathscr{G})\geq p|\mathcal{K}_k(\mathscr{G})|-\mu |V|^k
\]
holds for every family $\mathscr{G}=\{G_S:S\in \mathscr{S}\}$ associating with each $S\in \mathscr{S}$ some $G_S\subseteq V^S$.
\end{defi}

For example, if $k=3$ and $\mathscr{S}=\points =\big \{\{1\},\{2\}, \{3\} \big \}$, we can identify the sets $V^{\{1\}},V^{\{2\}},V^{\{3\}}\cong V$. If $H$ is $(p,\mu ,\points )$-dense, then for all sets $G_{\{1\}},G_{\{2\}},G_{\{3\}}\subseteq V$, there are at least
\[
p|G_{\{1\}}||G_{\{2\}}||G_{\{3\}}|-\mu |V|^3
\]
triples $(x,y,z)\in V^3$ such that $x\in G_{\{1\}}, y\in G_{\{2\}},z\in G_{\{3\}}$, and $\{x,y,z\}\in E$.
Similarly, for any integer $k\ge 2$, if we take $\mathscr{S} =\big \{\{1\}, \{2\},\dots, \{k\} \big \}$, then it is  convenient to identify the set
 $V^{\{i\}}\cong V$ for $i\in [k]$. These correspond to the definition of $(p,\mu )$-dense $k$-graphs from Definition \ref{def.}.

Now for $\mathscr{S}\subseteq \mathcal{P}([k])$ and $1\leq s \leq k-1$, define $\factor_k^{\mathscr{S} ,s}$ as the collection of $k$-graphs $F$ such that for all $0< p,\alpha<1$, there is some $n_0\in \mathbb{N}$ and $\mu > 0$ so that if $H$ is any $(p,\mu ,\mathscr{S})$-dense $k$-graph on $n$ vertices with $\delta _s(H)\geq \alpha n^{k-s}$, $n \geq n_0$ and $v(F)\mid n$, then $H$ has an $F$-factor.

The following problem generalizes Problem~\ref{prob}.
%It is natural to ask the following problem.
\begin{prob}\label{extend}
Given $\mathscr{S}\subseteq \mathcal{P}([k])$ and $1\leq s \leq k-1$, characterize the family $\factor_k^{\mathscr{S} ,s}$.
\end{prob}

%In general, it seems difficult to solve Problem~\ref{extend}.

\subsection{On generalizations of Theorem~\ref{k-graph}}
Throughout the rest of this section, write $\texttt{dots}:=\big \{\{1\}, \{2\},\dots, \{k\} \big \}$.
One can define $\cover_k^{\mathscr{S} ,s}$ similarly by replacing the property of having an ``$F$-factor'' by having an ``$F$-cover'' in the definition of $\factor_k^{\mathscr{S} ,s}$.
Our Theorem~\ref{k-graph} gives a reduction when $\mathscr{S} =\texttt{dots}$ and $s=1$, but actually the proof works for other (stronger) quasi-randomness as well, namely, we have $\factor_k^{\mathscr{S} ,1} = \cover_k^{\mathscr{S} ,1}$ for all $\mathscr{S}\subseteq \mathcal{P}([k])$.
However, we note that Theorem~\ref{k-graph} does not hold when $s>1$, even for the case $k=3$ and $\mathscr{S} =\texttt{dots}$.
For this, we shall describe another family of constructions.
%Since vertex degree is the weakest form of degree assumptions, if we consider the Problem~\ref{extend}, the proof of Theorem~\ref{k-graph} is helpful for attacking this problem with stronger quasi-randomness than $(p,\mu ,\points )$-denseness and other degree conditions. Recalling our proof, when $\mathscr{S} =\big \{\{1\}, \{2\},\dots, \{k\} \big \}$ and $s=1$, if $F\in \cover_k$, the proof of Lemma~\ref{Absorbing-Lemma} shows that the ``transferral" property holds and then we complete the absorption. However, for other quasi-randomness and degree conditions, it is usually not sufficient to verify ``transferral" property if we only have ``cover" property.

%We start with some quick notation.
Let us introduce some useful notation.
%Given a $k$-graph $F$ with $k\ge 2$, for a partition $\mathcal{P}=\{A,B\}$ of $V(F)$ and $e\in E(F)$, we write $\textbf{i} _{\mathcal{P},F}:=(|A|, |B|)$ and $\textbf{i} _{\mathcal{P},F}(e)=(|e\cap A|, |e\cap B|)$.
Given a $k$-graph $F$ with $k\ge 2$ and an integer $2\le s\le k-1$,
a bipartition $\mathcal{P}=\{A,B\}$ of $V(F)$ is \emph{$s$-shadow disjoint} if every two edges $e, e'\in E(F)$ with $\textbf{i} _{\mathcal{P}}(e)\neq \textbf{i} _{\mathcal{P}}(e')$ satisfy that $|e_1\cap e_2|<s$.
%We start with some quick notation.
%Given a $k$-graph $F$ with $k\ge 3$ and an integer $2\le s\le k-1$, a partition $\mathcal{P}=\{A,B\}$ of $V(F)$ is \emph{$s$-shadow disjoint} if every two edges $e_1, e_2\in E(F)$ with $|e_i\cap A|=a_i$, $|e_i\cap B|=b_i$ and $(a_1,b_1)\neq (a_2,b_2)$ satisfy that $|e_1\cap e_2|<s$.
%For $\mathcal{P}=\{A,B\}$ of $V(F)$, we write $\textbf{i} _{\mathcal{P},F}:=(|A|, |B|)$.
Then let $L^{s}_{F}$ be the lattice generated by all $\textbf{i} _{\mathcal{P}}(V(F))$ such that $\mathcal{P}$ is $s$-shadow disjoint and denote $\trans^{s}_k$ as the collection of $k$-graphs $F$ with $(1,-1)\in L^{s}_F$.

As an example, consider $F=K^{(3)}_{2,2,2}$, the complete $3$-partite $3$-graph with parts of size two, and we shall show that $K^{(3)}_{2,2,2}\notin \trans^2_3$.
Indeed, let $\mathcal{P}=\{A,B\}$ be a 2-shadow disjoint partition of $K^{(3)}_{2,2,2}$.
Take two edges $uvw, uvw'$ in $K^{(3)}_{2,2,2}$ so that $w$ and $w'$ are from the same part of the 3-partition.
Since these two edges share $s=2$ vertices, they must have the same index vector in $\mathcal{P}$, implying that $w$ and $w'$ must be both in $A$ or both in $B$.
Thus, we infer $(|A|,|B|)\in \{(0,6), (2,4), (4,2), (6,0)\}$, which implies that $(1,-1)\notin L^{2}_{F}$ with $F=K^{(3)}_{2,2,2}$.
On the other hand, it is not hard to prove that $K^{(3)}_{2,2,2}\in \cover_3^{\texttt{dots} ,2}$.
We shall prove in the following observation that $\factor_k^{{\tt dots} ,s} \subseteq \trans^{s}_k$.
Altogether we see that $K^{(3)}_{2,2,2}\notin \factor_3^{{\tt dots}, 2}$ (by $K^{(3)}_{2,2,2}\notin \trans^2_3$ and the observation) but $K^{(3)}_{2,2,2}\in \cover_3^{\texttt{dots} ,2}$, namely, Theorem~\ref{k-graph} does not hold if we replace the minimum degree condition by the minimum $s$-degree condition for $s>1$.
%Recall that Lenz and Mubayi~\cite{Lenz2016Perfect} constructed an $(n,\frac{1}{8},\mu)$ $3$-graph $H^*$ with $\delta _1(H^*)\geq (\frac{1}{8}-\mu )\binom{n}{2}$ and $H^*$ has no $K_{2,2,2}$-factor.

%In fact, it is easy to see that $H^*$ also satisfies $\delta_2(H^*)\ge (\frac{1}{8}-\mu)n$. Due to $\delta_2(H^*)\ge (\frac{1}{8}-\mu)n$, it is easy to verify that $H^*$ has a $K_{2,2,2,}$-cover.
%Inspired by this example, we describe another family of constructions.

%For $j \in [k-1]$, let $\mathbf{u}_j\in \mathbb{Z}^{k-1}$ be the $j$th unit vector, namely, $\mathbf{u}_j$ has $1$ on the $j$th coordinate and $0$ on other coordinates. A \textit{transferral} is a non-zero difference $\mathbf{u}_i-\mathbf{u}_j$ of $1$-vectors for $1\le i<j\le k-1$.

%The following result shows that for $k\ge 3$, $s\in [k-1]$ and $\mathscr{S} =\texttt{dots}$, $F\in \trans_k$ is necessary if $F$ satisfies the property in Problem~\ref{extend}.
%Now we prove the following observation.

\begin{observation}\label{dot-codegree}
For $2\le s\le k-1$,
%if a $k$-graph $F$ satisfies the property in Problem~\ref{extend} for $s\in [k-1]$ and $\mathscr{S} ={\tt dots}$, then $F\in \trans_k$.
$\factor_k^{{\tt dots} ,s} \subseteq \trans^{s}_k$.
\end{observation}

\begin{proof}
%Clearly, we have $\factor_k^{{\tt dots} ,s} \subseteq \cover_k^{\texttt{dots} ,s}$.
Let $2\le s\le k-1$ and $F \in \factor_k^{{\tt dots} ,s}$ with $f := v(F)$.
We will prove $F \in \trans^{s}_k$ by the following construction.
For $n\in f\mathbb{N}$, define a probability distribution $H(n)$ on $k$-graphs of order $n$ as follows.
Let $K$ be the complete $s$-graph on $n$ vertices, and $\mathcal{P}=\{X, Y\}$ be a partition of $V(K)$ satisfying $\textbf{i} _{\mathcal{P}}(V(K))=(n_1,n_2)$ with $n_1,n_2\ge n/3$.
Define $\phi:E(K) \longmapsto \{c_0,c_1, \dots, c_{k}\}$ as a random $(k-1)$-coloring with each color associated to an edge with probability $\frac{1}{k+1}$ independently.
We say a clique in $K$ is {\emph {$c_j$-colored}} if each edge in this clique is colored by $c_j$, where $0\le j\le k$.
Let now the vertex set of $H(n)$ be $V(K)$ and include a set $e\in [V(K)]^{k}$ in $E(H(n))$ if $K[e]$ is a $c_{j}$-colored clique for $0\le j\le k$ with $j= |e\cap X|$.
By concentration inequalities (e.g. the Janson inequality) and the union bound, for sufficiently large $n$,
there exists a $k$-graph $H\in H(n)$ such that $H$ is $(p,\mu ,\tt dots)$-dense with $\delta _s(H)=\Omega (n^{k-s})$, where $p=({k+1})^{-\binom{k}{s}}$.
Moreover, $\mathcal{P}$ is an $s$-shadow disjoint bipartition of $V(H)$ since for every two edges $e, e'\in E(H)$ with $\textbf{i} _{\mathcal{P}}(e)\neq \textbf{i} _{\mathcal{P}}(e')$, $K[e]$ and $K[e']$ must have different colors, which implies that $|e\cap e'|<s$.

By $n\in f\mathbb{N}$ and the definition of $H$, $H$ has an $F$-factor $\mathcal F$.
For every copy $F'$ of $F$ in $H$, as $F'$ is a subgraph of $H$, the bipartition $\mathcal{P}'$ of $V(F')$ inherited from the bipartition $\mathcal{P}$ is also $s$-shadow disjoint, implying that $\textbf{i}_{\mathcal P}(V(F'))=\textbf{i}_{\mathcal P'}(V(F'))\in L_F^s$.
Summing up the $\textbf{i}_{\mathcal P}(V(F'))$ over all $F'\in \mathcal F$ gives that  $(n_1,n_2)\in L^s_F$.

Analogously we can define a new $k$-graph $H'$ on $n$ vertices satisfying $(p,\mu ,\tt dots)$-denseness and $\delta _s(H')=\Omega (n^{k-s})$,
which has an $s$-shadow disjoint bipartition $\mathcal{P}^*$ with $\textbf{i} _{\mathcal{P^*}}(V(H'))=(n_1-1,n_2+1)$.
The same argument shows that $(n_1-1,n_2+1)\in L^s_F$.
Therefore we obtain $(1,-1)\in L^s_F$, that is, $F\in \trans^s_k$.
\end{proof}

We remark that it is possible to use less colors (e.g. $k-1$ colors) in the above proof, to push up the density of the construction.

%Let $\factor_3^{\points ,2}$ be the collection of $3$-graphs $F$ satisfying the property in Problem~\ref{extend}, and $\cover_3^{\points ,2}$ be the collection of $3$-graphs $F$ satisfying the following property: for all $0< p,\alpha <1$, there is some $n_0$ and $\mu> 0$ so that if $H$ is a $(p,\mu ,\points )$-dense $3$-graph on $n$ vertices with $\delta _2(H)\geq \alpha n$ and $n \geq n_0$, then each vertex $v$ in $V(H)$ is contained in a copy of $F$. Clearly, we have $\factor_3^{\points ,2}\subseteq \cover_3^{\points ,2}\bigcap \trans_3$.

In summary, for $2\le s \le k-1$, we have  $\factor_k^{{\tt dots} ,s} \subseteq \cover_k^{{\tt dots} ,s} \cap \trans^s_k$. %but conversely, it is not clear whether $ \cover_k^{{\tt dots} ,s} \cap \trans^s_k  \subseteq \factor_k^{{\tt dots} ,s}$ too.
%Moreover, when $s=1$, the reason for $\factor_k^{\texttt{dots} ,1} = \cover_k^{\texttt{dots} ,1}$
%is just because that for every $k$-graph $F$, if $F\in\cover_k^{\texttt{dots} ,1}$, then $F$ must contain a special vertex $v^*$ satisfing that $|S\cap S'|\leq 1$ for any $S\in N(v^*)$ and $S'\in \partial F'$ where $F'=F-\{v^* \}$ (which can be actually extracted from our proof). Therefore, it is clear that $\cover_k^{\texttt{dots} ,1} \subseteq \trans^s_k$ for all $2\le s \le k-1$.
%However, for $s>1$, it seems challenging to us to characterize $\factor_k^{\texttt{dots} ,s}$ explicitly. (One should note that although
%the minimum $s$-degree condition is stronger than the minimum $1$-degree condition,
%%$\cover_k^{\texttt{dots} ,1}\subseteq \cover_k^{\texttt{dots} ,s}$ is a smaller family,
%it does not make the characterization any easier.)

%\begin{prob}
%Characterize $\cover_3^{\texttt{dots} ,2}$.
%\end{prob}

\medskip
On a different direction, given a $k$-graph $F$, it is natural to ask for a characterization of the pairs $p, \alpha>0$ so that every $(n,p,\mu,\alpha)$ $k$-graph has an $F$-factor, and further one may replace $(p,\mu)$-denseness by stronger quasi-randomness.
Lenz and Mubayi~\cite{Lenz2016Perfect} constructed an $(n,\frac{1}{8},\mu )$ 3-graph $H^*$ with $\delta _1(H^*)\geq (\frac{1}{8}-\mu )\binom{n}{2}$, $\delta _2(H^*)\geq (\frac{1}{8}-\mu ){n}$ and $H^*$ has no $F$-factor for those $F$ that there exists a partition of the vertices of $F$ into pairs such that each pair has a common edge in their links.
In particular, $K^{(3)}_{2,2,2}$ satisfies this property.
A natural starting point would be to understand the above case when $F=K^{(3)}_{2,2,2}$.
\begin{prob}
Let $p,\alpha>1/8$ be given, and let $\mu$ be sufficiently small and $n$ be sufficiently large.
Suppose $H$ is an $(n,p,\mu)$ $3$-graph with $6\mid n$ and $\delta_1(H)\ge \alpha \binom n2$.
Does every such $H$ have a $K^{(3)}_{2,2,2}$-factor?
\end{prob}

\section*{Acknowledgements}
The authors extend gratitude to the referee for their careful
peer reviewing and helpful comments. 

%In particular, in a forthcoming work, we show that given any integers $a,b,c>0$, $p>1/8$ is enough to force the existence of a $K_{a,b,c}$-factor in every $(n,p,\mu,\alpha)$ $3$-graph.

%It is natural to ask the following question. If a $k$-graph $F$ violates the property in Problem~\ref{prob}, then whether there exsit constants $c$ so that for all $c<p<1$ and $0<\alpha <1$, there is some $n_0\in \mathbb{N}$ and $\mu > 0$ so that if $H$ is any $(n,p,\mu,\alpha)$ $k$-graph with $n \geq n_0$ and $v(F)\mid n$, then $H$ has an $F$-factor. If not, does $H$ have an $F$-factor if $H$ has stronger quasirandomness condition and degree condition?

\bibliographystyle{abbrv}
\bibliography{ref,ref1}

\end{document}